\documentclass[preprint]{imsart}

\RequirePackage{amsthm,amsmath,amsfonts,amssymb}
\RequirePackage[numbers]{natbib}
\RequirePackage{mathtools,dsfont}
\RequirePackage[colorlinks,citecolor=blue,urlcolor=blue]{hyperref}
\RequirePackage{graphicx}

\startlocaldefs
\numberwithin{equation}{section}
\theoremstyle{plain}
\newtheorem{theorem}{Theorem}[section]

\newtheorem{lemma}[theorem]{Lemma}
\newtheorem{proposition}[theorem]{Proposition}
\newtheorem{cor}[theorem]{Corollary}
\theoremstyle{remark}

\DeclarePairedDelimiter{\floor}{\lfloor}{\rfloor}
\DeclarePairedDelimiter\ceil{ \lceil}{ \rceil}
\newcommand{\N}{\mathbb{N}}
\newcommand{\R}{\mathbb{R}}

\newcommand{\mX}{\mathbb{X}}

\newcommand{\Z}{\mathbb{Z}}

\newcommand{\p}{\mathbb{P}}

\newcommand{\cA}{\mathcal{A}}

\newcommand{\cD}{\mathcal{D}}

\newcommand{\cF}{\mathcal{F}}

\newcommand{\cN}{\mathcal{N}}

\newcommand{\cM}{\mathcal{M}}

\newcommand{\cP}{\mathcal{P}}
\newcommand{\cR}{\mathcal{R}}

\newcommand{\cV}{\mathcal{V}}

\newcommand{\cX}{\mathcal{X}}
\newcommand{\fA}{\mathfrak{A}}

\newcommand{\fD}{\mathfrak{D}}

\newcommand{\fN}{\mathfrak{N}}
\newcommand{\fP}{\mathfrak{P}}

\newcommand{\fV}{\mathfrak{V}}
\newcommand{\E}{\mathbb{E}}
\newcommand{\V}{\operatorname{Var}} 
\renewcommand{\phi}{\varphi}
\renewcommand{\epsilon}{\varepsilon}
\newcommand{\diff}{\mathrm{d}}
\newcommand{\poi}{\operatorname{Poi}}

\newcommand{\ul}{\underline}
\newcommand{\ol}{\overline}
\newcommand{\wt}{\widetilde}
\newcommand{\wh}{\widehat}
\newcommand{\1}[1]{\,\mathds{1}\! \left\{ #1 \right\} }
\newcommand{\cPc}{\mathcal{P}^\circ}
\newcommand{\cPd}{\mathcal{P}^\dagger}
\newcommand{\Wc}{W^\circ}

\newcommand{\Mc}{M^\circ}
\newcommand{\dN}{\mathrm{N}}

\newcommand{\as}{\alpha_{stab}}
\endlocaldefs

\allowdisplaybreaks

\begin{document}

\begin{frontmatter}
\title{
On the Bahadur representation of sample quantiles for score functionals
}
\runtitle{Bahadur representation for score functionals}
\begin{aug}
\author{
\fnms{Johannes} \snm{Krebs}  
\ead[label=e1]{johannes.krebs@ku.de}
}
\address{KU Eichst{\"a}tt, Ostenstra\ss e 28, 85072 Eichst{\"a}tt, Germany.
\printead{e1} }

\runauthor{Johannes Krebs}
\affiliation{KU Eichst{\"a}tt}
\end{aug}

\begin{abstract}
We establish the Bahadur representation of sample quantiles for stabilizing score functionals in stochastic geometry and study local fluctuations of the corresponding empirical distribution function. The scores are obtained from a Poisson process. We apply the results to trimmed and Winsorized means of the score functionals and establish a law of the iterated logarithm for the sample quantiles of the scores.
\end{abstract}


%

%
% history:
% \received{\smonth{1} \syear{0000}}
%\tableofcontents
%
\end{frontmatter}

%%%
%%%
%%%
\section{Introduction}
Let $\{X_1,X_2,\ldots\}$ be an iid sequence of real-valued random variables with distribution function $F$. Let $p\in (0,1)$ and denote $\psi_p$ the $p$-quantile of $F$, i.e., $\psi_p = \inf\{ t \in \R: F(t) \ge p\}$. Let $F_n$ be the empirical distribution function of the first $n$ observations and $\psi_{p,n}$ the corresponding empirical $p$-quantile.

 Given that $F$ has at least two derivatives $f$ and $f'$ in a neighborhood of $\psi_p$ such that $f'$ is bounded in this neighborhood and that the density $f$ satisfies $f(\psi_p)>0$, the celebrated Bahadur representation \cite{bahadur1966note} states
 \begin{align}\label{E:BahadurOrig}
 	\psi_{p,n} = \psi_p + \frac{p - F_n (\psi_p)}{f(\psi_p)} + O_{a.s.}(n^{-3/4} (\log n)^{1/2} (\log \log n)^{1/4} ).
 \end{align}
Several refinements and extensions of this result have been pursued, with a particular attention on time series satisfying certain dependence patterns.

Kiefer's results (\cite{kiefer1967bahadur, kiefer1970deviations, kiefer1970old}) provided exact rates in the Bahadur representation for iid sequences. In the context of time series, extensions have been given in Sen \cite{sen1968asymptotic} for $m$-dependent processes, in Dutta and Sen \cite{dutta1971bahadur} for stationary autoregressive processes and in Sen \cite{sen1972bahadur} for mixing processes. The Bahadur representation for short-range dependent processes has been considered in Hesse \cite{hesse1990bahadur} and for long-range dependent processes in Ho and Hsing \cite{ho1996}. More recently, Wu \cite{wu2005bahadur} provided a representation for nonlinear time series and gave more details for short- and long-range dependent processes. Furthermore, Kulik \cite{kulik2007bahadur} considered weakly dependent linear processes.
Polonik \cite{polonik1997minimum} gives a Bahadur-Kiefer approximation for generalized quantile processes in the geometric context of minimum volume sets. Other contributions can be found in \cite{deheuvels1997strong, einmahl1996short, shang2010}.

In this manuscript we extend the Bahadur representation to so-called score functionals, which possess a stabilizing property. The latter quantifies the dependence within the observed data sample.

Many important functionals in stochastic geometry can be written as a sum of individual scores. Let $\cP$ be a Poisson process with unit intensity on $\R^d$, $d\ge 2$, defined on a generic probability space $(\Omega,\cF,\p)$. Let $W_n = [-n^{1/d}/2,n^{1/d}/2]^d$ be an observation window and let $\cP_n$ be the restriction of $\cP$ to $W_n$. We consider
\begin{align*}
	H(\cP_n) = \sum_{y\in\cP_n} \xi(y,\cP_n), \quad n\in\N,%,\\
	%H_n(\cX_n) = \sum_{y\in\cX_n} \xi(y,\cX_n), \quad n\in\N,
\end{align*}
where the scores $\xi(y,\cP_n)$ can be considered as the local contributions of the single points to the aggregated statistic $H$ provided the scores are stabilizing in a certain sense made precise below. The ideas of stabilizing score functionals in stochastic geometry date back to Penrose and Yukich \cite{penrose2001central, penrose2003weak} and have proven particularly useful to derive various types of limit theorems. Intuitively, stabilization can be related to the dependence between the single scores. Loosely speaking, sample data of fast stabilizing functionals corresponds for instance to times series data with a fast decaying dependence structure.
Central limit theorems \cite{baryshnikov2005gaussian, reitzner2013central, blaszczyszyn2019limit, flimmel2020limit} and normal approximations \cite{barbour2006normal, chatterjee2008new, lachieze2017new} have been studied in many variants in this context.

We establish the Bahadur representation for the scores $\{\xi(y,\cP_n): y \in \cPc_n\}$, where $\cPc_n$ (given in detail later) is a suitable subset of $\cP_n$, which omits points near the boundary. More precisely, we replace the initial iid sample $\{X_1,\ldots,X_n\}$ in \eqref{E:BahadurOrig} with the scores $\{\xi(y,\cP_n): y \in \cPc_n\}$ and along with a suitable empirical distribution function $\wh F_n$ and the corresponding empirical $p$-quantile $\wh \psi_{p,n}$, we give sufficient conditions for the following variant of \eqref{E:BahadurOrig}
\begin{align*}
 	\wh \psi_{p,n} = \psi_p + \frac{p - \wh F_n (\psi_p)}{f(\psi_p)} + O_{a.s.}(n^{-3/4} (\log n)^{5/4 + 3d/(4\alpha_{stab})} ).
\end{align*}
Here the parameter $\as >0$ describes the stabilizing behavior of the functional $\xi$. We have $\as = d$ in many applications in $d$-dimensional Euclidean space. This representation will even be uniform for $p$ in some compact interval. Moreover, the exponent in the log-factor can be reduced to $3/4 + 3d/(4\alpha_{stab})$ under a reasonable assumption, discussed below.

Apart from the stabilizing property of the scores, the most important features of the present setting are as follows. The sample size $\# \cPc_n$ is random as it is determined by the number of points in an observation window. The score functionals are translation invariant and typically non-linear. 

Plainly, the Bahadur representation requires the distribution of the score functionals to admit a positive density in a given neighborhood of a probability $p$. This rules out classical scores in topological data analysis such as the Euler characteristic or Betti numbers, which are integer valued and have a discrete distribution without a further modification. Potential scores in the present setting are volume and surface areas of the typical cell of a (Poisson-Voronoi) tessellation or statistics of convex hulls of samples of random points \cite{lachieze2019normal} as well as statistics of the minimal spanning tree \cite{steele1988growth,yukich2000asymptotics,chatterjee2017minimal} or k-nearest neighbor graph \cite{bickel1983sums}. We give some examples below.\\[5pt]

We conclude this section by introducing the necessary notation. 
Let $\dN$ be the set of $\sigma$-finite counting measures on $\R^d$ (``point configurations in $\R^d$''). So, we will treat elements from $\dN$ as sets in the following. For $\cM\in\dN$ and a Borel set $A\subseteq\R^d$, we write $\cM(A)$ for the measure of $A$ under $\cM$. We equip $\dN$ with the smallest $\sigma$-field $\fN$ such that for each Borel set $A\subseteq\R^d$ the evaluation map $m_A\colon \dN \to \N_0\cup \{ \infty \}; \cM \mapsto \cM(A)$ is measurable. Hence, a point process is a random element in $(\dN,\fN)$.

We write $\|\cdot\|$ for the Euclidean norm on $\R^d$ and $B(y,t) = \{z\in\R^d: \|z-y\| \le t\}$ for the closed $t$-neighborhood w.r.t.\ $\|\cdot\|$ of a point $y\in\R^d$ for $t\ge 0$. If $A\subseteq\R^d$ is Borel measurable, we write $|A|$ for its Lebesgue measure. If $A\subseteq\R^d$ is countable, we write $\# A$ for the number of its elements. Moreover, for a Borel set $A\subseteq \R^d$ and $z\in\R^d$, we write $A+z = \{x + z: x\in A\}$ for the set $A$ translated by $z$ and similarly $A-z = \{x - z: x\in A\}$. We write $Q_z = [z-1/2,z+1/2]^d$ for the $d$-dimensional cube with edge length 1 centered at $z\in\R^d$.

Given two sequences $(a_n)_n,(b_n)_n$, we write $a_n \lesssim b_n$ to indicate that $a_n\le C b_n$ for some universal constant $C>0$. The notation $a_n \asymp b_n$ is used to indicate that $a_n \lesssim b_n$ and $b_n \lesssim a_n$. We write $a_n \sim b_n$ if $a_n/b_n \to 1$.

For $x\in\R$, we write $\floor{x}$ (resp. $\ceil{x}$) for the maximal (minimal) integer not greater (not smaller) than $x$.

Let $(Z_n)_n$ be a sequence of random variables defined on a common probability space. Let $(u_n)_n\subseteq \R_+$ be a null sequence. If there is an $\epsilon_0 > 0$ such that $\p( \exists n |\forall m\ge n : |Z_m| \le \epsilon_0 u_m) =1$, we write $Z_n = O_{a.s.}(u_n)$ as $n\to\infty$. Moreover, if we have $\p( \exists n |\forall m\ge n : |Z_m| \le \epsilon_0 u_m) =1$ for all $\epsilon_0 > 0$, we write $Z_n = o_{a.s.}(u_n)$ as $n\to\infty$.

$\cP$ and $\cP'$ are homogeneous Poisson processes with unit intensity on $\R^d$ and are defined on $(\Omega,\cF,\p)$. $\cP$ and $\cP'$ are independent.

Let $\preceq$ denote the lexicographic ordering on $\Z^d$. For each $z\in\Z^d$, let $\cF_z$ denote the $\sigma$-field generated by $\cP$ restricted to the domain $\cup_{x\preceq z} Q_x$. So, $\cF_z$ is the smallest $\sigma$-field such that the number of points of $\cP$ (termed ``Poisson points'' in the following) in any bounded Borel set, which is contained in  $\cup_{x\preceq z} Q_x$, is measurable.

The space of c{\`a}dl{\`a}g functions on a given interval $I\subseteq \R$ is $D_I$. We write $\cD_I$ for the Borel-$\sigma$-field on $D_I$ generated by the Skorohod $J_1$-topology. $D_I$ is separable and topologically complete and $\cD_I$ coincides with the $\sigma$-field generated by the coordinate mappings, see \cite{bickel1971convergence, billingsley1968convergence}.

The rest of the paper is organized as follows. We introduce the setting in detail in Section~\ref{Section_Model} and present a uniform Bahadur representation in Section~\ref{Section_Results}, where we discuss potential examples, too. In Section~\ref{Section_Applications}, we establish local fluctuations of the empirical distribution function and -- based on the uniform Bahadur representation -- apply this result to trimmed and Winsorized means for samples of random size. Moreover, we present a law of the iterated logarithm for empirical quantiles of score functionals. All proofs are given in Section~\ref{Section_ProofsModel}, \ref{Section_Proofs}, \ref{Section_Supplement} and \ref{AppendixLemmas}.

\section{Distribution functions and quantiles of scores}\label{Section_Model}

The score functional $\xi$ is a measurable function from $\R^d \times \dN$ to $\R$ which is translation invariant in the sense that
\[
	\xi(y,P) = \xi(y-z,P-z)
	\]
for each $y,z\in\R^d$, $P\in\dN$. We set $\xi_r(y,P) \coloneqq \xi( y, P\cap B(y,r))$ for $r\ge 0$.

Consider the probability distribution function of $\xi$, when adding one additional point to the Poisson process $\cP$, viz.,
\begin{align}\label{E:Cdf}
	F(x) = \p( \xi(0,\cP\cup\{0\}) \le x ),
\end{align}
for $x\in\R$. We have 
$\p( \xi(y,\cP\cup\{y\}) \le x ) = \p( \xi(0,(\cP-y)\cup\{0\}) \le x ) = F(x)$ due to the translation invariance of the score functional and the Poisson process. So, the definition in \eqref{E:Cdf} is actually invariant under the choice of the additional point. This enables us to estimate $F$ by averaging over multiple observed scores (in a sense made precise below).

The ergodic behavior of this average is ensured by the following crucial property: The score functionals are strongly stabilizing in the sense that there is a \emph{radius of stabilization}, which is a measurable map $R: \R^d \times \dN\to\R$, with the property
\begin{align}\label{E:Stabilization1}
	\xi\big( y, (\cM\cup \{y\}\cup \cA ) \cap B(y,R(y,\cM\cup\{y\})) \big) =  \xi\big( y, \cM\cup \{y\}\cup \cA \big)
\end{align}
for all $y\in\R^d$, $\cM\in\dN$ and $\cA\subseteq \R^d$ finite.

We say the scores are \emph{exponentially stabilizing} on the Poisson process $\cP$ if there are $C_{stab},c_{stab},\as\in\R_+$ such that for $u \ge 0$,
\begin{align}\label{E:Stabilization2}
	\p( R(0, \cP\cup\{0\} ) \ge u) \le C_{stab} \exp( -c_{stab} u^{\as} ).
\end{align}

Apart from the distribution function $F$ given in \eqref{E:Cdf}, we also take into account the following variant, which captures second order effects. Let $Y$ be uniformly distributed on $B(0,\sqrt{d})$. Set
\begin{align}\label{E:DistributionExt}
	F^e(x) =	\p( \xi(0,\cP\cup\{0,Y\}) \le x)
\end{align}
for $x\in\R$. We will see later that the distribution function $F^e$ is only necessary in order to obtain a sharper result (improving the log-factor) in the Bahadur representation in  Theorem~\ref{T:BahadurPoissonUnif}. For more details we refer to the discussion following this theorem.

In order to establish the Bahadur representation for a given interval or point, we need that the distribution function $F$ admits a density $f$ which in turn admits a bounded derivative $f'$ in the corresponding neighborhood. Assuming that the distribution function $F^e$ has similar properties improves the rate in the Bahadur representation, see below.

In the following, let $W_n = [-2^{-1} n^{1/d}, 2^{-1} n^{1/d} ]^d$ be an observation window for  each $n\in\N$. We consider the restricted process $\cP_n = \cP|_{W_n}$. Similarly, throughout the manuscript let 
\begin{align}
\begin{split}\label{D:R}
	&r \coloneqq r_n  \coloneqq (c^* \log n)^{1/\as} \text{ for $n\in\N$ and for a constant } c^*  \ge 4/c_{stab}. 
\end{split}
\end{align}
The value of the constant $c^*$ can additionally depend on a specific statement, we clarify this where necessary. 

Define the smaller observation window
\[
	\Wc_n = [-(2^{-1} n^{1/d}-r), 2^{-1} n^{1/d} - r ]^d.
	\]
Then set $\cPc_n = \cP|_{\Wc_n}$.
We define $M_n = \# \cP_n$, resp. $\Mc_n = \# \cPc_n$ as the number of Poisson points inside $W_n$, resp. $\Wc_n$.

The condition in \eqref{E:Stabilization2} can be strengthened by requiring additionally: For all $n\in\N$, for all $y\in W_n$, for all $u\ge 0$
\begin{align}\label{E:Stabilization3}
	\p( R(y,\cP_n) \ge u ) \le C_{stab} \exp( - c_{stab} u^{\as} ).
\end{align}
Lachi{\`e}ze-Rey et al. \cite{lachieze2019normal} use a similar uniform condition on the stabilization of the scores $\xi$ and in many settings this stricter condition is satisfied.

We remark that the additional assumption \eqref{E:Stabilization3} is not necessary for the Bahadur representation in Theorem~\ref{T:BahadurPoissonUnif}. It is merely necessary in order to keep the radius $r$ growing at the minimal rate as in \eqref{D:R} independent of the dimension $d$ for the functional central limit theorem (Theorem~\ref{T:WeakConvergence}) and for the law of the iterated logarithm (Theorem~\ref{T:LILQuant}). Otherwise in these statements the constant $c^*$ additionally depends on the underlying dimension $d$.

When estimating the distribution function $F$ from a sample of Poisson points $\cP_n$, the radius of stabilization will be an important variable as the behavior of each score functional $\xi(y,\cP_n)$ crucially depends on the location of the Poisson point $y$, in particular on its proximity to the boundary of $W_n$. So, the value of a score, which is observed close to the boundary, is expected to be biased in a certain sense and it is quite natural to exclude such scores from the estimation procedure. For this purpose, we consider the following empirical distribution function which discards all points close to the boundary
\begin{align}\label{E:TrimmEcdf}
	\wh F_n\colon \R\to [0,1], x\mapsto 
	\frac{1}{\Mc_n} \sum_{y\in \cPc_n} \1{ \xi(y,\cP_n) \le x}.
\end{align}
Given that $\Mc_n =0$, the quotient $1/\Mc_n$ is $\infty$ but as the sum is empty, thus 0, the distribution function is $\wh F_n \equiv \infty \cdot 0 \equiv 0$ by the usual measure theoretic convention. In particular, the definition of $\wh F_n$ is meaningful.

Further, the event $\{ \Mc_n = 0\}$ has probability $\exp( - |\Wc_n| )$. So, the distribution function is non-trivial with a probability 
$1-\exp( - |\Wc_n| )$, in particular, the distribution function is eventually non-trivial for almost each state $\omega\in\Omega$. 
Finally, denote $\wh\psi_{p,n}$ the $p$-quantile of $\wh F_n$ for $p\in (0,1)$, i.e., $\wh\psi_{p,n} = \inf \{ t\in\R: \wh F_n(t) \ge p \}$.

Before, we start with the main results, we conclude this section with a result regarding the bias of $\wh F_n$.
\begin{theorem}[Bias of $\wh F_n$]\label{T:BiasWhF}
The bias of the estimator $\wh F_n$ vanishes polynomially. More precisely
$$
	\sup_{x\in\R} |\E[ \wh F_n(x) ] - F(x)| \lesssim  n^{-1/d} (\log n)^{1/\as} + n^{-1/2} (\log n)^{1/2} .
$$ 
\end{theorem}

%%%
\section{Main results}\label{Section_Results}

We state the first main result which is a uniform Bahadur representation for the sample quantile $\wh \psi_{p,n}$ of $\wh F_n$.

\begin{theorem}[Uniform Bahadur representation] \label{T:BahadurPoissonUnif}
Let the score functionals be translation invariant and exponentially stabilizing as in \eqref{E:Stabilization2} and let $r$ satisfy \eqref{D:R}.

Let $0<p_0\le p_1<1$ and $\epsilon>0$. Consider the interval $\fP=(\psi_{p_0} -\epsilon, \psi_{p_1}+\epsilon )$.
Let $\sup_{x\in \fP} |f'(x)| < \infty$ and let $\inf_{x\in \fP } f(x)>0$. 

\begin{itemize}
\item [(i)]
 Then for $\gamma=5/4$
\begin{align}\begin{split}\label{E:BahadurPoissonUnif0}
	&\sup_{p\in [p_0, p_1]} \Big| \wh \psi_{p,n} - \psi_p - \frac{p - \wh F_n(\psi_p)}{f(\psi_p)} \Big| \\
	&\qquad = O_{a.s.}\big( n^{-3/4} (\log n)^{\gamma+3d/(4\as)} \big).
\end{split}\end{align}
\\[-8pt]%
\item [(ii)]
Additionally assume that the distribution function $F^e$ from \eqref{E:DistributionExt} admits a density $f^e$ with $\sup_{x\in\fP} f^e <\infty$. Then \eqref{E:BahadurPoissonUnif0} is satisfied for $\gamma = 3/4$.
\end{itemize}
\end{theorem}
The additional requirements in part (ii) of Theorem~\ref{T:BahadurPoissonUnif} lead to an improvement of order $(\log n)^{-1/2}$ compared to part (i). From the technical perspective the improvement is due to the fact that the additional assumptions allow to bound above the second moment of a sum of local score functions of the type 
$$
	\E\Big[ \Big(\sum_{y\in \cP \cap Q_0} \1{\xi(y,\cP)\in I} \Big)^2 \Big]
$$
by a constant times the interval length of $I$, for $I$ being a bounded interval and $Q_0 = [-1/2,1/2]^d$. Since the location of the Poisson points $y$ of the scores $\xi(y,\cP)$ is random, we have to restrict ourselves to cubes of the type $Q_0$ in order to deal with the spatial dependence among the single scores. Hence, we always deal with a certain minimal interaction between score functionals. This is a major difference to discrete time series data $(X_k)_{k\in\Z}$ which comes naturally with a minimal granularity as it is indexed by the integers and thus, allows easier for splits.

In many applications the constant $\alpha_{stab}$ equals $d$, so that the exponent of the log-factor is either 2 in (i) or $3/2$ in (ii). It does not seem possible to remove the log-factor completely for the benefit of a log-log-factor given the present assumption of exponentially stabilizing score functionals. In the case of nonlinear time series data which satisfies a geometric-moment contraction, Wu \cite{wu2005bahadur} obtains the upper bound $O_{a.s.}(n^{-3/4} (\log n)^{3/2})$, too.

The main technical difference of our setting to that of Wu is that we have as well randomness in location of points, which is still present after a reduction to $m$-dependent score functionals. This spatial randomness is the reason, why we cannot rely in the proofs on classical uniform concentration for distribution functions.

\subsection{The density condition for the score functional $\xi$}

The Bahadur representation in Theorem~\ref{T:BahadurPoissonUnif} crucially relies on the density of the typical score $\xi(0,\cP\cup\{0\})$ and the variant $\xi(0,\cP\cup \{0,Y\})$, where $Y$ is uniformly distributed on $B(0,R)$ for some $R\in\R_+$.

First, we discuss relations between the scores $\xi(0,\cP\cup \{0,Y\})$ and $\xi(0,\cP\cup\{0\})$. Second, we present two classical examples. The first concerns scores on the $k$-nearest neighbors graph, the second scores on Poisson-Voronoi tessellations.

Set $\lambda = |B(0,R)| = w_d R^d$, where $w_d$ is the $d$-dimensional volume of the unit ball $B(0,1)$. For $j\in\N$ let $\mX_j$ be an iid sample of $j$ points on $B(0,R)$ having a uniform distribution and let $\cPd = \cP |_{B(0,R)^c}$. Obviously, we have for $I=(a,b)$
\begin{align}
		\p( \xi(0,\cP\cup\{0\}) \in I ) &= \sum_{j\ge 0} p_j q_j(I), \label{E:DistrXi}\\
		\begin{split}\label{E:DistrXi2}
		\p( \xi(0,\cP\cup\{0,Y\}) \in I ) &= \sum_{j\ge 0} p_j q_{j+1}(I) = \sum_{j\ge 0} \frac{j+1}{\lambda} p_{j+1} q_{j+1}(I)\\
		&= \sum_{j\ge 0} \frac{j}{\lambda} p_{j} q_{j}(I), 
		\end{split}
\end{align}
where $	p_j = e^{-\lambda} \frac{ \lambda^j }{j!}$ and $q_j(I) = \p( \xi(0, \{0\}\cup \mX_j  \cup \cPd ) \in I )$.

On the one hand, provided that $\xi(0,\cP\cup\{0\})$ admits a bounded density, $f$ say, it follows from \eqref{E:DistrXi} that each distribution $q_j$ also admits a density, $f_j$ say. Moreover, applying the Cauchy-Schwarz to \eqref{E:DistrXi2}, yields that $\xi(0,\cP\cup\{0,Y\})$ admits a density too, because we obtain then inequalities of the type
\begin{align*}
	\p( \xi(0,\cP\cup\{0,Y\}) \in I ) &\le \Big( \sum_{j\ge 0} \frac{j^2}{\lambda^2} p_j  \Big)^{1/2} \ \Big( \sum_{j\ge 0} p_j q_j(I) \Big)^{1/2} \\
	&= \sqrt{ \frac{1+\lambda}{\lambda}} \ \p( \xi(0,\cP\cup\{0\}) \in I )^{1/2}.
\end{align*}
However, owing to the additional factor $j$ in \eqref{E:DistrXi2}, it is not immediately clear whether or not this density is as well.

On the other hand, if there is a $\beta\in\R_+$ such that $0<q_j(I) \le \beta (1+j ^\beta) |I|$ for all $j\in\N_0$ and all intervals $I\subseteq \R$, then both scores admit a positive and bounded density function by the Radon-Nikod{\'y}m theorem. The polynomial bound on the conditional probabilities $q_j$ appears to be satisfied in many situations. We illustrate it below for statistics on the $k$-nearest neighbors graph.\\[5pt]

\textit{$k$-nearest neighbors graph.} 
Consider a configuration $\cX\in \dN$, $k\in\N$ and $x\in\cX$. The set of the $k$-nearest neighbors of $x$ in $\cX$ is $\fV_k(x,\cX)$, i.e., the $k$ points closest to $x$ in $\cX\setminus \{x\}$. We order the elements of $\fV_k(x,\cX)$ according to their distance to $x$ as follows $\|V_1(x,\cX)-x\| \le \|V_2(x,\cX) - x\|\le \ldots \le \|V_k(x,\cX) - x\|$. 

The (undirected) $k$-nearest neighbors graph with vertex set $\cX$ is obtained by including an edge $\{x,y\}$ if and only if $y\in \fV_k(x,\cX)$ or $x\in \fV_k(y,\cX)$. 

Consider the scores 
\begin{align*}
	\xi(x,\cX) &= \sum_{y\in \cX} \| x - y \| \1{ y\in \fV_k(x,\cX) }, \\
	\wt \xi(x,\cX) &= \|x-V_k(x,\cX) \|,
\end{align*} 
which give the total distance to the $k$-nearest neighbors resp.\ the distance to the $k$th-nearest neighbor of an $x\in\cX$. It is well-known that these functionals (and other functionals defined on the $k$-nearest neighbors graph) stabilize exponentially fast on a homogeneous Poisson process (with unit intensity on $\R^d$) in the sense of \eqref{E:Stabilization2} and \eqref{E:Stabilization3}, see the arguments given in \cite{penrose2001central, krebs2021LIL}. One has that $\as = d$, see also \cite{lachieze2019normal}.

We consider in the following the distribution of $\xi(0,\cP\cup \{0\} )$ and $\wt \xi(0,\cP\cup \{0\} )$. We write for short $V_i = V_i(0,\cP\cup \{0\})$. In order to derive the joint density of $(\|V_1\|,\ldots, \|V_k\|)$, it is sufficient to consider disjoint intervals
\[
	 [s_1, s_1+h_1], [s_2, s_2+h_2], \ldots, [s_k, s_k+h_k] 
\]
for $s_{i-1}+h_{i-1} < s_i$ and $\max_{1\le i\le k} h_i \downarrow 0$. (The probability that the modulus of two or more neighbors of 0 falls in the same interval $[s,s+\diff h)$ is zero by the properties of the Poisson distribution.) We have
\begin{align*}
	&\p( \|V_1\| \in [s_1,s_1+h_1], \|V_2\| \in [s_2,s_2+h_2], \ldots, \|V_k\| \in [s_k, s_k+h_k] )  \\
	&= \p( \cP( B(0,s_1) ) = 0) \cdot \p( \cP( B(0,s_1+h_1) \setminus B(0,s_1) ) = 1)   \\
	&\qquad \cdot \p( \cP( B(0,s_2) \setminus B(0,s_1+h_1) ) = 0) \cdot \p( \cP( B(0,s_2+h_2) \setminus B(0,s_2) ) = 1) \cdot \ldots \\
	&\qquad \ldots \cdot \p( \cP( B(0,s_{k-1}) \setminus B(0,s_{k-2}+h_{k-2}) ) = 0) \cdot \p( \cP( B(0,s_{k-1}+h_{k-1}) \setminus B(0,s_{k-1}) ) = 1) \\
	&\qquad \cdot \p( \cP( B(0,s_{k}) \setminus B(0,s_{k-1}+h_{k-1}) ) = 0) \cdot \p( \cP( B(0,s_k+h_k) \setminus B(0,s_k) ) \ge 1).
\end{align*}
And this equals
\begin{align*}
	& e^{-w_d s_1^d} \cdot \Big\{ w_d \big( (s_1+h_1)^d - s_1^d \big) e^{-w_d( (s_1+h_1)^d - s_1^d)} \Big\} \\
	& \cdot  e^{-w_d (s_2^d - (s_1+h_1)^d) } \cdot \Big\{ w_d \big( (s_2+h_2)^d - s_2^d \big) e^{-w_d( (s_2+h_2)^d - s_2^d)} \Big\} \cdot \ldots \\
	&  \ldots \cdot  e^{-w_d (s_{k-1}^d - (s_{k-2}+h_{k-2})^d) } \cdot \Big\{ w_d \big( (s_{k-1}+h_{k-1})^d - s_{k-1}^d \big) e^{-w_d( (s_{k-1}+h_{k-1})^d - s_{k-1}^d)} \Big\} \\
	& \cdot  e^{-w_d (s_k^d - (s_{k-1}+h_{k-1})^d) } \cdot \Big\{ 1 - e^{-w_d( (s_k+h_k)^d - s_k^d)} \Big\}.
\end{align*}
Consequently, the joint density of $(\|V_1\|,\ldots,\|V_k\|)$, which is supported on the set $\{ (s_1,\ldots,s_k) \subseteq \R^k: s_{i} < s_{i+1} \ \forall i\in \{1,\ldots,k-1\} \}$, is  
\begin{align*}
	&f(s_1,s_2,\ldots,s_k) \\
	&= e^{-w_d s_k^d } \ (w_d d)^k \ \Big( \prod_{i=1}^k s_i \Big) ^{d-1} \ \1{ s_1 < s_2 < \ldots < s_k}.
\end{align*}
In particular, the functional $\xi(0,\cP\cup\{0\})$ admits a continuous, a.e. positive and bounded density on $\R_+$ and, thus, satisfies the continuity assumption necessary for the Bahadur representation given in Theorem~\ref{T:BahadurPoissonUnif} (i).

We remark that one can infer from $f$ the joint density of the $k$-nearest neighbors $(V_1,\ldots,V_k)$ by choosing for a given realization $(s_1,\ldots,s_k)$ of the vector $(\|V_1\|,\ldots,\|V_k\|)$ for each $V_i$ a position on the $(d-1)$-sphere with radius $s_i$ according to a uniform distribution. However, we do not need this in the sequel.

We are rather interested in the continuity properties of the functional when adding another additional point $Y$ to $\cP$ which is uniformly distributed on $B(0,R)$ for some radius $R>0$. To simplify the analysis we consider here the functional $\wt\xi$, which only focuses on the $k$th-nearest neighbor.

In the above Poisson model, the $k$th-nearest neighbor has a probability distribution as follows.
\begin{align*}
	&\p( \|V_k\| \in [s,s+h]) = \p( \wt\xi( 0, \cP\cup \{0\}) \in [s,s+h]) \\
	&= \sum_{u=0}^{k-1} \p( \cP(B(0,s))= u, \cP( B(0,s+h) \setminus B(0,s) ) \ge k - u ) \\
	&= \sum_{u=0}^{k-1}  \frac{(w_d s^d)^{u} }{u!} e^{-w_d s^d} \  \sum_{\ell=k-u}^{\infty} \frac{(w_d( (s+h)^d - s^d))^\ell}{\ell!} e^{-w_d( (s+h)^d - s^d)}.
\end{align*}
Consequently, the density of $\wt\xi(0,\cP\cup\{0\})$ is
\begin{align*}
	g(s) &= \frac{(w_d s^d)^{k-1}}{(k-1)!} w_d d s^{d-1} e^{-w_d s^d} \1{s > 0} \\
	&= \frac{d\ w_d^k}{(k-1)!}  s^{dk-1} e^{-w_d s^d} \1{s > 0} .
\end{align*}
In particular, $g$ is continuous, a.e. positive and bounded on $\R_+$ and has a derivative $g'$ that is also bounded on $\R_+$. Consequently, the assumptions of Theorem~\ref{T:BahadurPoissonUnif} (i) are satisfied for this score functional.

Next, we show that the requirements in part (ii) of this theorem are also satisfied. For this purpose we consider the distribution of $\wt\xi(0,\cP\cup\{0,Y\})$, where $Y$ is uniformly distributed on $B(0,R)$. For this purpose, we use the notation from the beginning. Let $\cPd = \cP|_{B(0,R)^c}$ and let $\mX_j=\{X_1,\ldots,X_j\}$ be an iid sample of $j$ points on $B(0,R)$. We study the densities of $\wt\xi(0,\{0\}\cup \mX_j  \cup \cPd)$ for $j\in\N_0$. Plainly, the distance of 0 to the $k$th-nearest neighbor depends on $j$. 

If $j\ge k$, the $k$th-nearest neighbor lies inside $B(0,R)$. Hence, for $s<R$
\begin{align*}
	&\p( \wt\xi( 0, \{0\}\cup \mX_j \cup \cPd) \in [s,s+h] ) \\
	&= \sum_{u=0}^{k-1} \sum_{\ell = k-u}^{j-u} \p\Big( \{\text{$u$ $X_i$ are in $B(0,s)$} \} \cap \{\text{$\ell$ $X_i$ are in $B(0,s+h)\setminus B(0,s)$} \} \\
	&\qquad\qquad\qquad \cap \{\text{$(j-u-\ell)$ $X_i$ are in $B(0,R)\setminus B(0,s+h)$} \} \Big).
\end{align*}
And this last double sum is equal to
\begin{align*}
	& \sum_{u=0}^{k-1} \sum_{\ell = k-u}^{j-u} \frac{ j! }{ u! \ell! (j-u-\ell)! } \\
	&\qquad\qquad\quad \Big(\frac{ s^d }{ R^d}\Big)^{u} \Big(\frac{ (s+h)^d - s^d }{ R^d}\Big)^{\ell} \Big(\frac{ R^d - (s+h)^d }{ R^d}\Big)^{j-u-\ell} .
\end{align*}
So, for $j\ge k$ the corresponding density of $\wt\xi( 0, \{0\}\cup \mX_j \cup \cPd)$ on $\R^d$ is
\begin{align*}
	g_j(s) = \frac{j!}{ (k-1)! (j-k)!} \Big(\frac{ s^d}{R^d}\Big)^{k-1} d \frac{ s^{d-1}}{R^d} \Big(\frac{ R^d - s^d }{ R^d}\Big)^{j-k} \1{s\in (0,R)}.
\end{align*}
Note that $\sup_{s\in (0,R)} g_j(s) \le C j^k$ if $j\ge k$ for some $C\in\R_+$.

If $j<k$, the $k$th-nearest neighbor is outside of $B(0,R)$. Let $s > R$, then
\begin{align*}
	&\p( \wt\xi( 0, \{0\}\cup \mX_j \cup \cPd) \in [s,s+h] ) \\
	&= \sum_{u=0}^{k-1 - j} \sum_{\ell = k-u-j}^\infty \p( \{ \cPd( B(0,s)\setminus B(0,R)) = u \} \\
	&\qquad\quad \cap \{ \cPd( B(0,s+h) \setminus B(0,s)) = \ell \} ) \\ 
	&= \sum_{u=0}^{k-1 - j} \sum_{\ell = k-u-j}^\infty  \frac{(w_d(s^d-R^d))^{u}}{u!} e^{-w_d (s^d - R^d)} \\ &\qquad\qquad\qquad\qquad \frac{(w_d((s+h)^d-s^d))^{\ell}}{\ell!} e^{-w_d ((s+h)^d - s^d)} .
\end{align*}
So, for $j< k$ the corresponding density of $\wt\xi( 0, \{0\}\cup \mX_j \cup \cPd)$ on $\R^d$ is
\begin{align*}
	g_j(s) = \frac{(w_d(s^d - R^d))^{k-1-j}}{(k-1-j)!} e^{-w_d(s^d-R^d)} \ w_d d s^{d-1}   \1{s\in (R,\infty)}.
\end{align*}
Clearly, each $g_j$ is continuous and bounded and admits a bounded derivative $g'_j$.  Moreover, letting $\lambda = |B(0,R)|$, we have for the density of the score $\wt\xi(0,\cP\cup\{0\})$
\[
	g \equiv \sum_{j\ge 0} e^{-\lambda} \frac{\lambda^j}{j!} g_j.
\]
The density of the score $\wt\xi( 0, \cP\cup\{0,Y\} )$ is
\[
	\ol g = \sum_{j\ge 0} e^{-\lambda} \frac{\lambda^j}{j!} g_{j+1}.
\]
One finds that as well $\ol g$ is a.e. positive, bounded and continuous on $\R_+$ with a bounded derivative $\ol g'$.\\[5pt]

\textit{Poisson-Voronoi tessellations.}
Many deep results have been obtained for Voronoi tessellations with a particular emphasis when the underlying point process is a stationary Poisson process; for a short and incomplete list we mention -- apart from the results used below -- the contributions \cite{miles1982basic,brakke1987statistics,hayen2002areas,devroye2017measure}. However, only few results are available when it comes to exact distributional properties of statistics based on the Poisson-Voronoi tessellation.

The following example is based on the work of Zuyev \cite{zuyev1992estimates} and Hug et al. \cite{hug2004large}, who consider the limiting behavior of large Voronoi cells on a Poisson input.

The exponential stabilization of the Voronoi tessellation on the homogeneous Poisson process $\cP$ is well-known, for examples illustrating the planar case of dimension 2, see \cite{mcgivney1999asymptotics, penrose2001central}. One can follow the arguments of Penrose \cite{penrose2007laws} to see that the exponent is $\as = d$.

Given $\cP$, denote $Z = C(0, \cP\cup \{0\} )$ the typical Voronoi cell. We measure the deviation of $Z$ from a ball with center $0$ with the following statistic
\[
	\delta(Z) = \frac{R_0 - \rho_0}{R_0 + \rho_0},
\]
where $R_0$ is the radius of the smallest ball with center 0 containing $Z$ and $\rho_0$ is the radius (possibly 0) of the largest ball with center 0 contained in $Z$.
The volume of $Z$ is $v_d(Z)$.

In the following, let $\epsilon\in (0,1)$ and let $\sigma_0>0$. We consider the score functional
\begin{align*}
	\xi(0,\cP\cup \{0\}) &= v_d( Z) \1{ \delta(Z) \ge \epsilon }.
	\end{align*}
Following Hug et al. \cite{hug2004large}, we have for $h\in (0,1/2)$ and $a\ge \sigma_0>0$
\begin{align*}
	\p( \xi(0,\cP\cup \{0\}) \in (a,a+h) ) =
	&\p( v_d(Z) \in a(1,1+h), \delta(Z) \ge \epsilon) \\
	&\le c_1 h \exp \big\{ - \big( 1+c_2 \epsilon^{(d+3)/2} \big) 2^d  a \big\} 	
\end{align*}
for two constants $c_1,c_2$, which depend on $\sigma_0,\epsilon$ and $d$ but not on $a,h$.
In particular, $\xi(0,\{0\}\cup \cP)$ admits a density on $\R_+$, which is bounded on compact intervals.

Another interesting statistic is the volume of the \textit{fundamental region} of $Z$. This statistic dates back at least to Zuyev \cite{zuyev1992estimates}, see also this contribution for an illustration of the fundamental region. Let $Z$ be formed by the Poisson points $Y_1,\ldots,Y_M$, where the number $M$ is random. 

For each 0-face of $Z$ there are $d+1$ points $\{0,Y_{i_1},\ldots,Y_{i_d}\}$ which are equidistant to it. So for each 0-face there is a ball centered at it with $\{0,Y_{i_1},\ldots,Y_{i_d}\}$ on its boundary and no Poisson points in it.  Each $Y_i$ ($1\le i\le M$) lies in the intersection of $d$ such balls. The union of all these balls for all points forming $Z$, is the fundamental region of $Z$, $\cR(Z)$. We consider the score
\[
		\ol \xi(0, \cP\cup\{0\}) = v_d( \cR(Z)).
\]	

Conditional on the event that the number of hyperfaces of $Z$ is $m$, Zuyev \cite{zuyev1992estimates} finds that $v_d( \cR(Z))$ follows a $\Gamma(m,1)$-distribution (remember that $\cP$ has unit intensity) with density 
\[
	h_m(s) = ( (m-1)! )^{-1} s^{m-1} \exp(-s),
\]
with $h_m$ being defined for all $m\in\N_+$ in the following.

In particular, the density of $\ol \xi(0, \cP\cup\{0\})$, $h$ say, satisfies
\[
		0 < h = \sum_{m\ge d+1} p_m h_m,
\]	
where for $m \ge d+1$
\[
	p_m = \p( Z \text{ has $m$ hyperfaces} ) >0.
\]
Clearly, $h$ is locally bounded away from 0 and $\infty$ because for each $s>0$
\[
	0< p_{d+1} h_{d+1}(s) \le h(s) \le \sup_{m\ge d+1} h_m(s) = e^{-s} \sup_{m\ge d+1} \frac{ s^{m-1}}{(m-1)!}\le 1.
\]
Moreover, the derivative $h'$ satisfies for each $s>0$
\begin{align*}
	h'(s) &= \sum_{m\ge d+1} p_m h_{m-1}(s)\Big(1-\frac{s}{m-1}\Big) \\
		&= \sum_{m\ge d+1} p_m h_{m}(s)\Big(\frac{m-1}{s}-1\Big).
\end{align*}
Thus, $h'$ is locally bounded because 
\begin{align*}
	| h'(s)| &\le \sum_{m\ge d+1} p_m \min \Big\{ h_{m-1}(s)\Big|1-\frac{s}{m-1}\Big|, h_{m}(s)\Big|\frac{m-1}{s}-1\Big| \Big\}\\
	&\le \sum_{m\ge d+1} p_m \big(h_{m-1}(s) \1{s\le m-1} + h_m(s) \1{s > m-1} \big) \\
	&\le \sum_{m\ge d+1} \big[ p_m \sup_{m\ge d} h_m(s) \big] \le \sup_{m\ge d} h_m(s)\le 1.
\end{align*}
Hence, the score functional $\ol\xi$ satisfies the assumptions of Theorem~\ref{T:BahadurPoissonUnif} (i).

\section{Applications}\label{Section_Applications}

We begin with a general result on the weak convergence of the empirical distribution function of the observed scores.
\begin{theorem}[Weak convergence in $D$]\label{T:WeakConvergence} Let the score functional be translation invariant and exponentially stabilizing as in \eqref{E:Stabilization2}. Let the choice of $r$ satisfy \eqref{D:R} such that additionally the condition in \eqref{E:Stabilization3} is satisfied or that additionally $c^* > (2d+1+1/d)/c_{stab}$.
\begin{itemize}
	\item [(i)] Let $I=[a,b]\subseteq \R$. Let $F$ be differentiable on $I$ with $\sup_{x\in I} f(x)<\infty$. Then the sequence $( \sqrt{n}(\wh F_n(x) - F(x)): x\in I)_n$ converges weakly to some centered Gaussian process $(W(x):x\in I)$ in the Skorokhod space $(D_I,\cD_I)$.
	
	Moreover, $W$ has a continuous modification which has H{\"o}lder continuous sample paths of any exponent $\beta\in (0,1/2)$. \\[-4pt] 
	\item [(ii)] Let $0<p_0\le p_1<1$ and let $I=[\psi_{p_0}, \psi_{p_1}]$. Assume that the conditions of Theorem~\ref{T:BahadurPoissonUnif} (i) are satisfied for the extended interval $(\psi_{p_0}-\epsilon,\psi_{p_1}+\epsilon)$ for some $\epsilon>0$. Then $(\sqrt{n}(\wh \psi_{p,n} - \psi_p): p_0 \le p \le p_1)_n$ converges weakly to $(W(x)/f(x) :x\in I)$ in $(D_I,\cD_I)$.
\end{itemize}
\end{theorem}
%%%
%%%
%%%
An analytical expression of the covariance structure of $W$ can be derived with Proposition~\ref{P:AbstractLIL}. In the special case of the variance, we provide an analytic result in \eqref{E:nu}, which provides information about the positivity of the limit variance, too.

Using the continuous mapping theorem, we can conclude the next corollary. A similar result was obtained in Wu \cite{wu2005bahadur} for dependent real-valued sequences. 
Given $\cP_n$, consider the order statistics of the scores $\xi(y,\cP_n)$, whose points $y$ lie inside the observation window $\Wc_n$, viz.,
\[
	\wh\psi_{1/\Mc_n, n} \le \wh\psi_{2/\Mc_n, n} \le \ldots \le \wh\psi_{1, n}.
	\]
Let $0<p_0<p_1<1$ and set $\alpha_n=\floor{\Mc_n p_0}$, $\beta_n=\floor{\Mc_n p_1}$. The trimmed and Winsorized means are of the form
\begin{align*}
	\ol \xi_{t,n} &= (\beta_n - \alpha_n)^{-1} \ \sum_{i={\alpha_n}+1}^{\beta_n} \wh\psi_{i/\Mc_n, n},  \\
	\ol \xi_{w,n} &= {\Mc_n}^{-1} \Big( \alpha_n \wh\psi_{\alpha_n /\Mc_n, n} + \sum_{i={\alpha_n}+1}^{\beta_n} \wh\psi_{i /\Mc_n, n}  + (\Mc_n-\beta_n) \wh\psi_{ (\beta_n+1) /\Mc_n, n}   \Big) .
\end{align*}
Then we have the following corollary.

\begin{cor}[Trimmed and Winsorized means]\label{Cor:Means}
Assume that the conditions of Theorem~\ref{T:WeakConvergence} (i) and (ii) are satisfied for $I=[\psi_{p_0}, \psi_{p_1}]$. Then there are $\sigma_t,\sigma_w<\infty$ such that the trimmed mean $\ol \xi_{t,n}$ and the Winsorized mean $\ol \xi_{w,n}$ satisfy
\begin{align}
	\sqrt{n} \Big\{	\ol \xi_{t,n} - (p_1-p_0)^{-1} \int_{p_0}^{p_1} \psi_{u} \ \diff u	\Big\}	\Rightarrow \cN(0,\sigma_t^2), \label{E:Trimmed} \\
	\sqrt{n} \Big\{	\ol \xi_{w,n} - p_0 \psi_{p_0} - \int_{p_0}^{p_1} \psi_{u} \ \diff u - (1-p_1) \psi_{p_1}	\Big\}	\Rightarrow \cN(0,\sigma_w^2). \label{E:Winsorized}
\end{align}
\end{cor}

\begin{proof}[Proof of Corollary~\ref{Cor:Means}]
For the trimmed mean $\ol \xi_{t,n}$ the case is as follows. By the choice of $\Wc_n$, we have $\Mc_n/n\to 1$ $a.s.$ Now, we split the left-hand side in two terms.

First, we use that $\wh\psi_{u,n}$ is non-decreasing to see that
\[
	\int_{\alpha_n / \Mc_n}^{\beta_n / \Mc_n} \wh \psi_{u,n} \ \diff u \le {\Mc_n}^{-1} \sum_{i=\alpha_n+1}^{\beta_n}	\wh \psi_{i/\Mc_n, n} \le \int_{(\alpha_n+1) / \Mc_n}^{(\beta_n +1) / \Mc_n} \wh \psi_{u,n} \ \diff u .
\]
Moreover, using Theorem~\ref{T:BahadurPoissonUnif}, we see $\sup_{\alpha_n/\Mc_n \le u \le (\beta_n+1)/\Mc_n} \wh \psi_{u,n} = O_{a.s.}(1)$. Consequently,
\begin{align}\label{E:Means1}
	(\beta_n - \alpha_n)^{-1} \ \sum_{i={\alpha_n}+1}^{\beta_n} \wh\psi_{i/\Mc_n, n} - (p_1 - p_0)^{-1} \int_{p_0}^{p_1} \wh \psi_{u,n} \ \diff u = O_{a.s.}(n^{-1}).
\end{align}
Second, by Theorem~\ref{T:WeakConvergence} $( \sqrt{n}(\wh\psi_{p,n}-\psi_p): p_0\le p\le p_1 )$ converges weakly to a Gaussian process $W$. So applying the continuous mapping theorem, there is a $\sigma_t< \infty$ such that
\begin{align}\label{E:Means2}
	\sqrt{n} \ (p_1 - p_0)^{-1} \int_{p_0}^{p_1} \wh \psi_{u,n} - \psi_{u} \ \diff u \Rightarrow \cN(0,\sigma^2_t).
\end{align}
Now, \eqref{E:Trimmed} follows from \eqref{E:Means1} and \eqref{E:Means2}.
The case for the Winsorized mean in \eqref{E:Winsorized} is very similar, we omit the details.
\end{proof}

Moreover, we can recover a pointwise law of the iterated logarithm for quantiles. For this purpose define $\cP''_0 = (\cP \setminus Q_0) \cup (\cP'\cap Q_0)$. Set for $p\in (0,1)$
\begin{align}
	\nu_p^2 &= \E\Big[ \E\Big[ \sum_{y\in \cP \cap Q_0} ( \1{ \xi(y,\cP) \le \psi_p} - p) - \sum_{y\in \cP'\cap Q_0 } (\1{ \xi(y, \cP''_0) \le \psi_p}  - p) \nonumber \\
	&\qquad\qquad + \sum_{y\in \cP\setminus Q_0} (\1{ \xi(y,\cP) \le \psi_p} - \1{ \xi(y, \cP''_0) \le \psi_p})  \Big| \cF_0 \Big]^2 \Big]. \label{E:nu}
\end{align}

First, note that the definition of $\nu_p^2$ is meaningful because by the result in Lemma~\ref{L:DeltaInfty} the third sum in the conditional expectation is actually taken over finitely many points $y\in \cP\setminus Q_0$ only. Moreover, for $\nu_p^2$ to be positive, Penrose and Yukich \cite{penrose2001central} give a sufficient condition in terms of the add-one cost function. More precisely, given the exponential stabilization of the score functionals from \eqref{E:Stabilization2}, we deduce that the following random variable exists for each $x\in\R$
\begin{align}\begin{split}\label{E:AddOneCost}
	\fD_\infty(x) &= \Big(\1{\xi(0,\cP\cup\{0\}) \le x } - F(x) \Big) \\
	&\quad +\sum_{y\in\cP} \Big(\1{\xi(y,\cP\cup\{0\}) \le x }  - \1{\xi(y,\cP) \le x} \Big),
\end{split}\end{align}
see again Lemma~\ref{L:DeltaInfty}. Given that the distribution of $\fD_{\infty}(\psi_p)$ is non-degenerate, the variance $\nu_p^2$ is positive, see Penrose and Yukich \cite[Theorem 2.1]{penrose2001central}.

Relying on the general LIL of Krebs \cite{krebs2021LIL}, we conclude this section with a LIL for the empirical quantiles $(\wh\psi_{p,n})_n$ given a $p\in (0,1)$. A more general and abstract result which covers as well the LIL for the empirical distribution function is given in Proposition~\ref{P:AbstractLIL}.

\begin{theorem}[LIL for quantiles]\label{T:LILQuant}
Let the score functional be translation invariant and exponentially stabilizing as in \eqref{E:Stabilization2}. Let the choice of $r$ satisfy \eqref{D:R} such that additionally condition~\eqref{E:Stabilization3} is satisfied or assume further that $c^* > (2d+1+1/d)/c_{stab}$.

Let $p \in (0,1)$ and let $f(\psi_p)>0$ and $f'$ be bounded in a neighborhood of $\psi_p$. Assume that $\fD_\infty(\psi_p)$ is non-degenerate. Then we have
\begin{align*}%\label{E:LILQuant}
	\limsup_{n\to\infty} \frac{\pm \sqrt{n} (\wh \psi_{p,n} - \psi_p) }{\sqrt{2 \log \log n} } = \frac{\nu_p}{f(\psi_p)} > 0 \quad a.s.
\end{align*}
\end{theorem}

\section*{Acknowledgments}
The author gratefully acknowledges the support of the German Research Foundation (DFG), grant number KR 4977/2-1.

\section{Proofs of the result in Section~\ref{Section_Model}}\label{Section_ProofsModel}

As proposed in \cite{flimmel2020limit}, we can consider another estimator of $F$ which relies solely on the stabilized scores $\xi$ and not as $\wh F_n$ on those scores which are sufficiently far away from the boundary of the observation window. We shall call this estimator $\wt F_n$, it turns out to be nearly unbiased and it allows us to quantify the bias of $\wh F_n$.

To this end, define $B^*(y,\cP)$ as the stabilizing ball of the functional $\xi$ at $y\in\cP$ given the entire Poisson process $\cP$, this is
$$
	B^*(y,\cP) = B(y, R(y,\cP)).
$$
Next, define for two sets $A,B\subset \R^d$ the erosion of $A$ by $B$ as the set 
$$
	A\ominus B=\{ x\in\R^d: x+B \subseteq A\}.
$$
The following identities hold for a translation $z\in\R^d$
\begin{align*}
	A \ominus (B-z) = z + (A\ominus B) = (z+A)\ominus B.
\end{align*}
Hence, $A$ and $B$ are Borel measurable, we have for the Lebesgue measure of the erosion $|A\ominus B| = |A \ominus (B-z) | = |(z+A)\ominus B|$.

Then we define
\begin{align}\label{E:wtF}
	\wt F_n\colon\R\to\R, \ x\mapsto \sum_{\substack{y\in\cP_n: \\ y \neq 0} } \frac{ \1{B^*(y,\cP) \subseteq W_n } }{ |W_n \ominus B^*(y,\cP) | } \ \1{ \xi(y,\cP_n) \le x}.
\end{align}
To see that the definition is meaningful, note that by the above
\begin{align*}
	|W_n \ominus B^*(y,\cP) | > 0 &\quad \Leftrightarrow\quad  | W_n \ominus (y + B(0, R(y,\cP)) ) | > 0  \\
	  &\quad \Leftrightarrow\quad |  W_n \ominus B(0, R(y,\cP)) | > 0 \\
	  &\quad \Leftrightarrow\quad  R(y,\cP)) < n^{1/d}/2;
\end{align*}
where the last equivalence is true because $W_n$ is a cube of edge length $n^{1/d}$ while $ B(0, R(y,\cP))$ is a ball of radius $R(y,\cP)$ centered at the origin. Furthermore, the condition $B^*(y,\cP) \subseteq W_n$ in the definition \eqref{E:wtF} implies $R(y,\cP) \le n^{1/d}/2$ and if equality holds, then $y$ is equal to the origin, which is an element of $\cP$ with probability 0.

%Clearly, if $R(y,\cP)) < n^{1/d}/2$, we have the strict subset relation $B(0, R(y,\cP)) \subsetneq W_n$. However, due to the geometry of the ball and cube, this strict subset property is also satisfied if $R(y,\cP)) = n^{1/d}/2 $. Hence, $B^*(y,\cP) \subseteq W_n$ and $|W_n \ominus B^*(y,\cP)| = 0$ can occur (with probability 0).

%Here recall the measure theoretic definition $\infty \cdot 0 = 0$, hence, the quotient of the indicator function and the erosion is meaningful.

Note that the estimator $\wt F_n$ does not necessarily satisfy $\lim_{x\to\infty} \wt F_n(x) = 1$. In particular, the last limit  can be smaller or greater than 1.

Moreover, the estimator $\wt F_n$ is not necessarily measurable w.r.t.\ the $\sigma$-field generated by $\cP_n= \cP|_{W_n}$, whereas the estimator $\wh F_n$ is. E.g., consider for instance the situation on the $k$-nearest neighbors graph for a point $y\in\cP_n$ close to the boundary of the observation window $W_n$. Here $\xi$ could be the total distance to the $k$-nearest neighbors of $y$ in $\cP$. Depending on the configuration of $\cP$ outside of $W_n$ the $k$-nearest neighbors of $y$ can all lie inside or all outside of $W_n$ (or in parts inside $W_n$ and parts outside of it). Clearly, this uncertainty only dissolves when observing $\cP$ on a sufficiently large ball around $y$. Hence, from the practical point of view when being confronted only with the realization of $\cP_n$ and not more, one prefers the estimator $\wh F_n$ over $\wt F_n$. This is why we will restrict ourselves in the following to the properties of $\wh F_n$. We only use $\wt F_n$ to derive an upper bound for the bias of $\wh F_n$, this is motivated by the fact that the bias of $\wt F_n$ decreases at an exponential rate.

\begin{lemma}[Bias of $\wt F_n$]\label{L:BiasWtF}
The bias of the estimator $\wt F_n$ vanishes exponentially. More precisely 
$$
	\E[ \wt F_n(x)] - F(x) = -\E\Big[ \1{R(0,\{0\}\cup \cP) \ge n^{1/d}/2} \1{\xi(0,\{0\}\cup\cP) \le x} \Big] 
$$ for all $x\in\R$, $n\in\N$.
\end{lemma}

\begin{proof}[Proof of Lemma~\ref{L:BiasWtF}]
Define $h(y,\cP_n) = \1{ \xi(y,\cP_n) \le x}$ for $x\in\R$, $n\in\N$ arbitrary but fixed.
Now, let $X'$ be an independent random variable with uniform distribution on $W_n$. Then we have by the Slyvniak-Mecke formula
\begin{align}
	&\E\left[ \sum_{0 \neq y\in\cP_n}  \frac{\1{B^*(y,\cP)\subseteq W_n} }{ |W_n \ominus B^*(y,\cP) | }	 h(y,\cP_n) \right] \nonumber \\
	&=  \E\left[ \sum_{0 \neq y\in\cP_n}  \frac{\1{B^*(y,\cP)\subseteq W_n} }{ |W_n \ominus B^*(y,\cP) | }	 h(y,\cP) \right] \nonumber \\
	&= |W_n|  \ \E\left[\frac{\1{B^*(0,\{0\}\cup (\cP-X'))\subseteq (W_n- X')} }{ |(W_n - X') \ominus B^*(0,\{0\}\cup (\cP-X') ) | }  \1{X'\neq 0} h(0,\{0\}\cup (\cP-X') )  	\right] \nonumber \\
	&= |W_n| \ \E\left[ \frac{\1{B^*(0,\{0\}\cup \cP)\subseteq W_n - X'} }{ |(W_n - X') \ominus  B^*(0,\{0\}\cup \cP ) | } \1{X' \neq 0}  h(0,\{0\}\cup \cP )	\right] \nonumber \\
	\begin{split}\label{E:BiasWtF1}
	&= |W_n| \ \E\Bigg[ \frac{\1{B^*(0,\{0\}\cup \cP)\subseteq W_n - X'} }{ |(W_n - X') \ominus  B^*(0,\{0\}\cup \cP ) | }  \1{X' \neq 0}  \\
	&\qquad\qquad\qquad\1{ |(W_n - X') \ominus  B^*(0,\{0\}\cup \cP ) | > 0} h(0,\{0\}\cup \cP )	\Bigg];
	\end{split}
\end{align}
here the second to last equality follows from the stationarity of the Poisson process and the independence of $X'$ and where the last equality follows because $B^*(0,\{0\}\cup \cP ) \subseteq W_n - X'$ and $X'\neq 0$ imply a positive Lebesgue measure for the erosion $(W_n-X') \ominus B^*(0,\{0\}\cup \cP )$.

Since the volume of the erosion is invariant under translations, we have 
\begin{align*}
	\eqref{E:BiasWtF1} &= |W_n| \ \E\Bigg[ \frac{\1{B^*(0,\{0\}\cup \cP)\subseteq W_n - X'} }{ |W_n  \ominus  B^*(0,\{0\}\cup \cP ) | }  \1{X' \neq 0}  \\
	&\qquad\qquad\qquad\1{ |W_n \ominus  B^*(0,\{0\}\cup \cP ) | > 0} h(0,\{0\}\cup \cP )	\Bigg] \\
	&=  \E\Bigg[ \frac{\int_{W_n \setminus \{0\} } \1{B^*(0,\{0\}\cup \cP)\subseteq W_n - x } \diff x}{ |W_n  \ominus  B^*(0,\{0\}\cup \cP ) | } \\
	&\qquad\qquad\qquad \1{ |W_n \ominus  B^*(0,\{0\}\cup \cP ) | > 0} h(0,\{0\}\cup \cP ) 	\Bigg]  \\
		&=  \E\Bigg[ \frac{\int_{W_n\setminus \{0\}} \1{x\in W_n \ominus B^*(0,\{0\}\cup \cP)} \diff x}{ |W_n  \ominus  B^*(0,\{0\}\cup \cP ) | }  \\
	&\qquad\qquad\qquad \1{ |W_n \ominus  B^*(0,\{0\}\cup \cP ) | > 0} h(0,\{0\}\cup \cP ) 	\Bigg]  \\
	&= \E\Bigg[ \frac{|W_n \cap (W_n \ominus B^*(0,\{0\}\cup \cP) ) | }{|W_n  \ominus  B^*(0,\{0\}\cup \cP ) |}  \\
	&\qquad\qquad\qquad \1{ |W_n \ominus  B^*(0,\{0\}\cup \cP ) | > 0} h(0,\{0\}\cup \cP ) \Bigg] \\
	%& = \E\Bigg[  \1{ |W_n \ominus  B^*(0,\{0\}\cup \cP ) | > 0} h(0,\{0\}\cup \cP ) \Bigg] \\
	&= \E\Bigg[  \1{ R(0,\{0\}\cup \cP) < n^{1/d}/2 } h(0,\{0\}\cup \cP ) \Bigg],
\end{align*}
where we use $W_n \ominus B^*(0,\{0\}\cup \cP) \subseteq W_n$ for the last equality. 
\end{proof}

\begin{proof}[Proof of Theorem~\ref{T:BiasWhF}]
We use the estimator $\wt F_n$ to facilitate the computations. The bias of $\wt F_n$ vanishes exponentially uniformly in $x\in\R$ by Lemma~\ref{L:BiasWtF}. 

Again define $h(y,\cP_n) = \1{ \xi(y,\cP_n) \le x}$ for $x\in\R$, $n\in\N$ arbitrary but fixed. As the origin is an element of the Poisson process with probability 0, we have
$$
	\E[ \wh F_n(x) ] = \E\Bigg[ \frac{1}{\Mc_n} \sum_{0\neq y\in \cPc_n} h(y,\cP_n) \Bigg].
$$
Consequently, we have
\begin{align}
	\begin{split}\label{E:BiasWhF1}
	&\E[ \wh F_n(x) - \wt F_n(x) ]\\
	&= \E\Bigg[ \sum_{0\neq y\in\cPc_n }  h(y,\cP_n) \1{ B^*(y,\cP)\subseteq W_n }  \left( \frac{1}{\Mc_n} - \frac{1}{ | W_n \ominus B^*(y,\cP) | } \right) \Bigg] 
	\end{split} \\
	&\quad +  \E\left[ \sum_{0\neq y\in\cPc_n }  h(y,\cP_n) \frac{\1{ B^*(y,\cP) \not \subseteq W_n }}{\Mc_n} \right] \label{E:BiasWhF2} \\
	&\quad  +  \E\left[ \sum_{0 \neq y\in\cP_n\setminus \cPc_n }  h(y,\cP_n) \frac{\1{ B^*(y,\cP) \subseteq W_n }}{ | W_n \ominus B^*(y,\cP) |  } \right]. \label{E:BiasWhF3}
\end{align}
We show in the sequel that the terms in \eqref{E:BiasWhF1} resp.\ in \eqref{E:BiasWhF3} on the right-hand side determine the upper bound of the bias, while the term in \eqref{E:BiasWhF2} is negligible.

Using the Slyvniak-Mecke formula, the term in \eqref{E:BiasWhF2} is at most
\begin{align}
	 &\E\left[ \sum_{y\in\cPc_n } \frac{\1{ B^*(y,\cP) \not \subseteq W_n }}{\Mc_n} \right] \le  \E\left[ \sum_{y\in\cPc_n } \frac{\1{ R(y,\cP) \ge r }}{\Mc_n} \right] \nonumber \\
	 &= | \Wc_n | \ \E\left[ \frac{\1{R(0,\{0\}  \cup (\cP-X')  ) \ge r} }{\# \cPc_n + 1} \right] \label{E:BiasWhF4}
\end{align}
where $X'$ is independent of $\cP$ and uniformly distributed on $\Wc_n$.
Now, $\E[ (\#\cPc_n + 1)^{-2} ] \le 2/ |\Wc_n|^2$. Moreover,
$$
	\p( R(0, \{0\} \cup \cP) \ge r) \le C_{stab} e^{-c_{stab} r^{\as} } \lesssim n^{-4}.
$$
Consequently, applying Hölder's inequality to the term in \eqref{E:BiasWhF4}, we see that the term in \eqref{E:BiasWhF2} is of order $n^{-2}$ uniformly in $x$.

The term in \eqref{E:BiasWhF3} is bounded above as follows: Clearly, $y\in\cP_n \setminus \cPc_n$ implies $0\neq y$. Applying the Slyvniak-Mecke formula yields
\begin{align*}
	&\E\left[ \sum_{y\in\cP_n\setminus \cPc_n }  h(y,\cP_n) \frac{\1{ B^*(y,\cP) \subseteq W_n }}{ | W_n \ominus B^*(y,\cP) |  } \right] \\
	&= \E\left[  h(y,\cP_n) \frac{\int_{W_n\setminus \Wc_n} \1{B^*(0,\{0\}\cup \cP) \subseteq W_n - z} \diff z }{| W_n \ominus B^*(0,\{0\}\cup \cP) | } \right] \\
	&= \E\left[  h(y,\cP_n) \frac{ |(W_n\setminus \Wc_n) \cap ( W_n\ominus B^*(0,\{0\}\cup \cP)) |  }{| W_n \ominus B^*(0,\{0\}\cup \cP) | } \right].
\end{align*}
If $R(0,\{0\}\cup \cP) \le r$, then
\begin{align} \label{E:BiasWhF7}
	| W_n \ominus B^*(0,\{0\}\cup \cP) | = (n^{1/d}-2R(0,\{0\}\cup \cP))^d .
\end{align}
Hence,
the quotient in the last expectation is at most (omitting the erosion in the numerator)
\begin{align*}
	\frac{n - (n^{1/d} - 2r )^d}{(n^{1/d}-2R(0,\{0\}\cup \cP))^d } \le \frac{n - (n^{1/d} - 2r )^d}{(n^{1/d}-2r)^d } 
\end{align*}
and otherwise the quotient is zero.
Consequently, the last expectation and the term in \eqref{E:BiasWhF3} are of order $n^{-1/d} r \asymp n^{-1/d} (\log n)^{1/\as}$.

The term in \eqref{E:BiasWhF1} features the main part of the bias. It is sufficient to consider the two expectations
\begin{align}
	&\E\left[ \sum_{0\neq y\in\cPc_n} \frac{\1{B^*(y,\cP)\subseteq W_n}}{|W_n \ominus B^*(y,\cP)| } \frac{| |W_n \ominus B^*(y,\cP)| - |\Wc_n| | }{ \Mc_n } \right] \label{E:BiasWhF5} \\
	&\quad + \E\left[ \sum_{0\neq y\in\cPc_n} \frac{\1{B^*(y,\cP)\subseteq W_n}}{|W_n \ominus B^*(y,\cP)| } \frac{ | |\Wc_n| - \Mc_n | }{ \Mc_n } \right] \label{E:BiasWhF6}
\end{align}
which is an upper bound for the expectation in \eqref{E:BiasWhF1}.

We begin with the expectation in \eqref{E:BiasWhF5}. Then applying the Slyvniak-Mecke formula and relying on the volume formula for the erosion in \eqref{E:BiasWhF7}
\begin{align}
	&\E\left[ \sum_{0\neq y\in\cPc_n} \frac{\1{B^*(y,\cP)\subseteq W_n}}{|W_n \ominus B^*(y,\cP)| } \frac{| |W_n \ominus B^*(y,\cP)| - |\Wc_n| | }{ \Mc_n } \right] \nonumber \\
	\begin{split}\label{E:BiasWhF8}
	&= | \Wc_n| \ \E\Biggl[ \frac{\1{X'\neq 0} \1{B^*(0,\{0\}\cup \cP) \subseteq W_n - X' } }{|W_n \ominus B^*(0,\{0\}\cup \cP)| }  \\
	&\qquad\qquad\qquad\qquad \frac{ | (n^{1/d} - 2R(0,\{0\}\cup \cP))^d - (n^{1/d}-2r)^d | }{ \cP( \Wc_n - X') + 1 } \Biggl],\end{split}
\end{align}
where the independent random variable $X'$ is uniformly distributed on $\Wc_n$.

Next, we derive an upper bound for the second quotient inside the expectation in \eqref{E:BiasWhF8} and this bound does not depend on $X'$. To this end, we partition $\Wc_n$ into its $2^d$ subcubes $\Wc_{n,j}$, $1\le j\le 2^d$, each of which satisfies
$$
	\Wc_{n,j} = \Wc_n \cap (\Wc_n + e)
$$
for a unique $e\in \{-1,1\}^d$. Then for the denominator 
\begin{align}\label{E:BiasWhF9}
	\frac{1}{ \cP( \Wc_n - X') + 1 } \le \sum_{j=1}^{2^d} \frac{1}{ \cP( \Wc_{n,j} ) +1 } \quad a.s.
\end{align}
Moreover using the abbreviation $R=R(0,\{0\}\cup \cP)$, the numerator satisfies
\begin{align*}
	g_n(r,R) &:= | (n^{1/d} - 2R(0,\{0\}\cup \cP))^d - (n^{1/d}-2r)^d | \\
	&\lesssim n^{(d-1)/d} (r+R) + n^{(d-2)/d} R^2 + \ldots + R^d
\end{align*}
because $r$ grows at a logarithmic scale. Thus, relying in the first step on \eqref{E:BiasWhF9} and then integrating w.r.t.\ the independent $X'$, \eqref{E:BiasWhF8} is at most
\begin{align}\label{E:BiasWhF10}
	\sum_{j=1}^{2^d} \E\left[ \frac{g_n(r,R)}{ \cP( \Wc_{n,j} ) +1 } \right] \le \sum_{j=1}^{2^d} \E[ g_n(r,R)^2 ]^{1/2}  \ \E\left[ \frac{1}{ (\cP( \Wc_{n,j} ) + 1)^2 } \right]^{1/2}.
\end{align}
Clearly, $|\Wc_{n,j}|$ is of order $n$, thus, $\E[ (\cP( \Wc_{n,j} ) + 1)^{-2} ] \lesssim n^{-2}$. Moreover, using the rate of decay from \eqref{E:Stabilization2}, the radius of stabilization satisfies for all $1\le k \le 2d$
\begin{align*}
	&\E[ R(0,\{0\}\cup \cP)^k ] = \int_0^\infty \p\Big(  R(0,\{0\}\cup \cP) \ge t^{1/k} \Big) \diff t \\
	&\le C_{stab} \int_0^\infty \exp( - c_{stab} t^{\as/(2d)} ) \diff t.
\end{align*}
This shows that \eqref{E:BiasWhF10} is of order $n^{-1/d} (\log n)^{1/\as}$ and consequently, \eqref{E:BiasWhF5} converges to zero at least at the same rate.

Finally, we consider the expectation in \eqref{E:BiasWhF6} which equals
\begin{align}\label{E:BiasWhF13}
	| \Wc_n| \ \E\Biggl[ \frac{ \1{X'\neq 0} \1{B^*(0,\{0\}\cup \cP) \subseteq W_n - X' } }{|W_n \ominus B^*(0,\{0\}\cup \cP)| }  \frac{ \big| \cP( \Wc_n - X') + 1 - |\Wc_n| \big| }{ \cP( \Wc_n - X') + 1 } \Biggl].
\end{align}
Again, we remove $X'$ from the second quotient inside the expectation by taking the maximum over all positions $y\in \Wc_n$, then remove the first quotient through integration and finally, we apply H{\"o}lder's inequality to separate the numerator and denominator of the second quotient. So, \eqref{E:BiasWhF13} is bounded above by
\begin{align}
\begin{split}\label{E:BiasWhF14}
	& \E\Biggl[ \frac{1}{ \min_{y\in \Wc_n} (\cP( \Wc_n - y) + 1)^2 }\Biggl]^{1/2} \\
	&\qquad\qquad\qquad  \cdot \E\Biggl[ \max_{y\in \Wc_n} \big| \cP( \Wc_n - y) + 1 - |\Wc_n| \big|^2 \Bigg]^{1/2}.
	\end{split}
\end{align}
Once more, we can again rely on the uniform upper bound from \eqref{E:BiasWhF9} for the first expectation in \eqref{E:BiasWhF14}. Moreover, we develop an approximation for the numerator as follows: First we discretize  with a suitable grid. Define $\ol W_n = \{\Wc_n + 2y: y\in Q_0\}$ and set $B_n = \ol W_n \cap \Z^d$. Then $\# B_n \lesssim n$. Furthermore, abbreviate the compound Poisson $\cP( \Wc_n - y) - |\Wc_n|$ by $\ol \cP( \Wc_n - y)$. For $y\in B_n$, define a $d$-dimensional cube with upper corner $y$ as follows: 
$$
	\ul Q_y = \{x\in \Wc_n: x \le y \text{ and } x \not < y' \text{ for all } y' \in B_n, y'\le y\}.
	$$
(Given two generic $x,y\in\R^d$, we use the notation $x\le y$ (resp. $x<y$) if and only if $x_i\le y_i$ (resp. $x_i<y_i$) for all $1\le i\le d$.) Then
\begin{align}
	\max_{y\in \Wc_n}  | \ol \cP( \Wc_n - y) | &\le \max_{y\in B_n} | \ol \cP( \Wc_n - y) |\label{E:BiasWhF11} \\
	&\quad +  \max_{y\in B_n} \max_{x\in \ul Q_y } | \ol \cP( \Wc_n - x) - \ol \cP( \Wc_n - y)  |. \label{E:BiasWhF12}
\end{align}
The term in \eqref{E:BiasWhF11} can be treated with the concentration properties. We have for each $c'\in\R_+$
\begin{align*}
	&\E[ \max_{y\in B_n} | \ol \cP( \Wc_n - y) |^2 ]^{1/2} \\
	&\le \E\Big[ \1{| \ol \cP( \Wc_n - y) | >  c' n^{1/2} (\log n)^{1/2} \text{ for one } y \in B_n } \\
	&\qquad\qquad\qquad \cdot \max_{y\in B_n} | \ol \cP( \Wc_n - y) |^2  \Big]^{1/2} + c' n^{1/2} (\log n)^{1/2} \\
	&\le \# B_n^{1/4} \ \max_{y\in B_n} \p\Big(  | \ol \cP( \Wc_n - y) | >  c' n^{1/2} (\log n)^{1/2} \Big)^{1/4} \\
	&\qquad \cdot \# B_n^{1/4} \ \max_{y\in B_n} \E[ | \ol \cP( \Wc_n - y) |^4]^{1/4} + c' n^{1/2} (\log n)^{1/2}.\end{align*}
Clearly, if we choose $c'$ sufficiently large, this last term is of order $n^{1/2} (\log n)^{1/2}$ by Lemma~\ref{Lemma0}.

The term in \eqref{E:BiasWhF12} can be treated in the same way, using additionally the monotonicity of the Poisson counting measure: Let $\ul y \in \Z^d$ be the point closest to $y$ such that $\ul y_i < y_i$ for each $1\le i\le d$. Then
\begin{align*}
	&\max_{x\in \ul Q(y) } | \ol \cP( \Wc_n - x) - \ol \cP( \Wc_n - y)  |  \\
	&\le \cP( (\Wc_n - \ul y)\setminus (\Wc_n -  y) )  +  \cP( (\Wc_n - y)\setminus (\Wc_n - \ul y) ) \\
	&\le \ol \cP( (\Wc_n -  y)\triangle (\Wc_n - \ul y) ) + | (\Wc_n - y)\triangle (\Wc_n - \ul y) |
\end{align*}
with $ | (\Wc_n - y)\triangle (\Wc_n - \ul y) |$ being of order $n^{(d-1)/d}$. Reasoning as before, one finds that
$$
	\E[ \max_{y\in B_n} \ol \cP( (\Wc_n - \ul y)\triangle (\Wc_n -  y) )^2]^{1/2} = O( n^{(d-1)/(2d)} (\log n)^{1/2} ).
$$
Hence, \eqref{E:BiasWhF14} is of order $n^{-1/2} (\log n)^{1/2}$. This completes the proof.
\end{proof}

\section{Proofs of the results in Section~\ref{Section_Results}}\label{Section_Proofs}

The results rely on approximations of the score functionals. For this purpose, define the truncated scores
\[
	\xi_r(y,P\cup \{y\} ) = \xi(y, (P\cup \{y\})\cap B(y,r))
	\]
for the choice $r = r_n  = (c^* \log n)^{1/\as}$ for a $c^*  \ge 4/c_{stab}$ as in \eqref{D:R}. We remark that the value of the constant $c^*$ can also depend on a specific statement, we clarify this where necessary. 

Then we consider for $x\in\R$
\begin{align*}
	\wh F_{n,r}(x) &=  \frac{1}{\Mc_n} \sum_{y\in \cPc_n} \1{ \xi_r (y,\cP_n) \le x} , \\
	\wh H_{n,r}(x) &= \frac{\Mc_n}{|\Wc_n|} \wh F_{n,r}(x). 
\end{align*}

Moreover, put $\cPd_n = \cP_n \setminus \cPc_n$, which is defined on $W_n \setminus \Wc_n$. So similarly as $\cPc_n$, $\cPd_n$ depends on the choice of $r$, too. Note that by definition
\[
	|\Wc_n | = n - 2 d r n^{(d-1)/d} + (C_d + o(1)) r^2 n^{(d-2)/d}, \quad n\to\infty, 
\]
for some $C_d\in\R$. In particular, $|\Wc_n| / |W_n| \to 1$ as $n\to\infty$.

Set $B_n = \{ z\in \Z^d : Q_z \cap W_n \neq \emptyset \}$, where $Q_z = z + [-1/2,1/2]^d$ for $z\in\R^d$.
Recall that $\cP'$ is another Poisson process with unit intensity on $\R^d$ such that $\cP$ and $\cP'$ are independent. Set $\cP'_n = \cP'_n |_{W_n}$. For $z\in\R^d$ define
\begin{align}
	\cP''_z &= (\cP\setminus Q_z) \cup (\cP' \cap Q_z), \label{D:PP1}\\
	\cP''_{n,z} &= (\cP_n\setminus Q_z) \cup (\cP'_n \cap Q_z) , \label{D:PP2} \\
	\cPc_{n,z} &= (\cP''_{n,z})|_{\Wc_n}.  \label{D:PP3}
\end{align}

First, we give an exponential inequality for score functionals which possess a deterministic and finite radius of stabilization on the homogeneous Poisson process.

\begin{proposition}\label{P:ExpIneqPoisson}
Let $I=(a,b)\subseteq\R$. %Let $r = (c^* \log n)^{1/\as}$ for $c^* \ge 4/c_{stab}$.
\begin{itemize}
	\item [(i)] Let $F$ be differentiable on $I$ such that $\sup_{x\in I} f(x) < \infty$. For all $\tau>1$ there is a $C_\tau>0$ such that for any positive, bounded sequence $(a_n)_n$ with the property $(\log n)^{2+d/\as} = o(a_n n)$, we have for $\gamma=2$
\begin{align}\begin{split}\label{E:ExpIneqPoisson0}
	&\sup_{\substack{x,y\in I\\ |x-y|\le a_n }} \ \p\Big(  \big| (\wh H_{n,r}(x) - \wh H_{n,r}(y)) - \E\big[\wh H_{n,r}(x) - \wh H_{n,r}(y) \big]	\big| \\
	&\qquad\qquad\ge C_\tau \Big( \frac{(\log n)^{\gamma+d/\as} a_n}{n} \Big)^{1/2} \Big) = O( n^{-\tau} ).
	\end{split}
\end{align}\\[-8pt]
	\item [(i')] Additionally to (i) assume that $F^e$ from \eqref{E:DistributionExt} is differentiable on $I$ such that its density $f^e$ satisfies $\sup_{x\in I} f^e(x) <\infty$. Then for all $\tau>1$ there is a $C_\tau>0$ such that for any positive, bounded sequence $(a_n)_n$ with $(\log n)^{1+d/\as} = o(a_n n)$ the result in \eqref{E:ExpIneqPoisson0} holds for $\gamma=1$.\\[-4pt]
	\item [(ii)] For all $\tau>1$ there is a $C_\tau>0$ such that
\begin{align*}
	&\sup_{\substack{x\in I}} \ \p\Big(  \big| \wh H_{n,r}(x)  - \E\big[\wh H_{n,r}(x) \big]	\big| \ge C_\tau \Big( \frac{(\log n)^{1+d/\as}}{n} \Big)^{1/2} \Big) = O( n^{-\tau} ).
\end{align*}
\end{itemize}
\end{proposition}

\begin{proof}
Let $\tau>1$ be arbitrary but fixed.
We begin with the proof of (i); (i') and (ii) work then similarly.
Let $J\subseteq \R$ be an interval.
Set $d_n = (\ol c \log n) \vee 2$ for $\ol c = 8(\tau+2)$ sufficiently large and
$$
	A_{n,z} = \{ \cP(Q_z) \le d_n \}, \quad z\in\R^d.
$$
Then by Lemma~\ref{Lemma0}
\begin{align*}
	\p( A_{n,z}^c ) = \p( \cP( Q_0) > d_n ) \le 2  \exp\Big( - \frac{(d_n - 1)^2}{2d_n} \Big).
\end{align*}
Let $\ul z$ be the element in $B_n$ which satisfies $\ul z \preceq y$ for all $y\in B_n$. Set
\[
			B_{n,1} = \{ \ul z + \ceil{2r+\sqrt{d}+1} (k_1,\ldots,k_d)^T \in B_n : k_1,\ldots,k_d \in \N_0 \},
\]
where the superscript $v^T$ denotes the transpose of a vector $v\in\R^d$.

Set $S=\ceil{2r+\sqrt{d}+1}^d$.
Let $B_{n,2},\ldots,B_{n,S}$ be the $S-1$ translations of the type $s+ B_{n,1}$ of $B_{n,1}$ for $s \in \{0,1,\ldots,\ceil{2r+\sqrt{d}+1} -1\}^d \setminus \{0\} $. Then $B_{n,1},\ldots,B_{n,S}$ partition $B_n$.

Define for $z\in B_n$
\begin{align*}
	\Xi_{r}(z,n) &= \sum_{y\in \cPc_n } \1{ \xi_r(y,\cP_n) \in J} \1{y\in Q_z} \ge 0, \\
	\Xi_{1,r}(z,n) &= \Xi_{r}(z,n) \1{A_{n,z}} \text{ as well as }	\Xi_{2,r}(z,n) = \Xi_{r}(z,n) \1{A_{n,z}^c }.
\end{align*}
Then set
\[
	\Sigma_1 =\sum_{j=1}^S \sum_{z\in B_{n,j}} \Xi_{1,r}(z,n) \text{ and } \Sigma_2 = \sum_{j=1}^S \sum_{z\in B_{n,j}} \Xi_{2,r}(z,n).
\]
We have
\[
		\sum_{y\in \cPc_n } \1{ \xi_r(y,\cP_n) \in J} =  \Sigma_1 + \Sigma_2.
\]

First, we consider $|\Wc_n|^{-1} (\Sigma_2-\E[\Sigma_2])$ and show that this term is negligible. Let $\epsilon_0 > 0$, then we obtain the following upper bound which is uniform over all intervals $J$
\begin{align*}
	&\p( | \Sigma_2 - \E[\Sigma_2] | \ge \epsilon_0 n^{-2} |\Wc_n| )	\\
	&\le 2 \frac{n^2}{\epsilon_0 |\Wc_n|} \ \E[ \Sigma_2 ] \\
	&\le 2 \frac{n^2}{\epsilon_0 |\Wc_n|}  \sum_{j=1}^S \sum_{z\in B_{n,j}} \E[ \cPc_n(Q_z) \1{A_{n,z}^c} ] \\
	&\lesssim  n^2 \E[ \cPc_n(Q_0)^2]^{1/2} \p( A_{n,0}^c )^{1/2} \\
	& \lesssim n^2 \exp\Big( - \frac{(d_n - 1)^2}{4d_n} \Big).
\end{align*}
In particular, choosing $\ol c$ sufficiently large as above, we find in (i) and (ii)
$$
	\p\big( |\Wc_n|^{-1} \big| \Sigma_2 - \E[\Sigma_2] \big| \ge \epsilon_0 n^{-2} \big) = O( n^{-\tau} ).
$$
This shows that $|\Wc_n|^{-1} (\Sigma_2 - \E[\Sigma_2])$ is negligible.

We come to $\Sigma_1 - \E[\Sigma_1]$. The crucial observation is now that $\{ \Xi_{1,r}(z,n): z \in A \}$ and $\{ \Xi_{1,r}(z,n): z\in A' \}$ are independent whenever $\|z - z'\|_\infty \ge 2r+\sqrt{d}+1$ for all $z\in A$ and $z'\in A'$ because the functionals $\xi_r$ act in the $r$-neighborhood of a Poisson point $y\in Q_z$, resp. $y'\in Q_{z'}$, and  
\[
	\| y - y' \| \ge \|z-z' \| - \|y - z\| - \|y' - z' \| \ge \|z - z'\|_\infty  - 2 \frac{\sqrt{d}}{2} \ge 2r+1 > 2r.
\]

Set $\sigma_{n,z}^2 = \V( \Xi_{1,r}(z,n) )$. 
Depending on the setting, we choose the sequence $(u_n)_n$ suitably, see below. Then we use that $\Xi_{1,r}(z,n) \le d_n$ to apply the Bernstein inequality as follows
\begin{align}
	&\p\big( \big| \Sigma_1 - \E[\Sigma_1] \big| \ge  u_n  |\Wc_n| \big) \nonumber \\
	&\le \sum_{j=1}^{S } \p\Big( \Big| \sum_{z\in B_{n,j}}  \Xi_{1,r}(z,n)- \E[ \Xi_{1,r}(z,n) ] \Big| \ge  u_n |\Wc_n|/S \Big) \nonumber \\
	&\le 2  \sum_{j=1}^{S } \exp\bigg\{ - \frac{1}{2} \frac{ ( u_n  |\Wc_n| /S )^2 }{ \sum_{z\in B_{n,j} } \sigma_{n,z}^2 + d_n  u_n  |\Wc_n| / (3 S ) }	\bigg\},  \label{E:ExpIneqPoisson11}
\end{align}
We have $\# B_{n,j} \lesssim n r^{-d}$ both in (i), (i') and in (ii). Regarding the variances, we have different estimates in (i), (i') and (ii), which then in turn lead to different rates.

We begin with (i), where the interval $J\subseteq I$ satisfies $|J| = |x-y| \le a_n$ and where we set
$$
	u_n = \epsilon_1 ( (\log n)^{2+d/\as} a_n / n)^{1/2}.
$$
Let $q = \ul c \log n$, for $\ul c > 0$ sufficiently large. We apply Lemma~\ref{Lemma3} in combination with Lemma~\ref{Lemma0} to obtain the following uniform upper bound
\begin{align}
	\sigma_{n,z}^2 &\le \E\Big[ \Xi_{1,r}(z,n)^2 \Big]   \nonumber\\
	&\lesssim (1+q) \ \p( \xi_r(0,\cP \cup \{0\} ) \in J ) +  2 \exp\Big( - \frac{\floor{q}^2}{2(\floor{q}+1)} \Big)  \nonumber\\
	&\lesssim \log n \ ( \p( \xi(0,\cP \cup \{0\} ) \in J ) + n^{-4} ) \nonumber \\
	&\lesssim \log n \ (a_n + n^{-4}) \label{E:ExpIneqPoisson13};
\end{align}
the second to last inequality holds because $r \asymp (\log n)^{1/\as}$ and because when switching from $\xi_r(0,\cP \cup \{0\} )$ to $\xi(0,\cP \cup \{0\} )$ the error is of order $e^{-c_{stab}r^{\as}} \le n^{-4}$. For the last inequality, we use that the probability is of order $|J|$, which is uniformly bounded above by $a_n$.

Consequently in (i), using \eqref{E:ExpIneqPoisson11} there is a $c_1\in\R_+$, which does not depend on the choice of $(a_n)_n$ and $\epsilon_1$ such that
\begin{align*}
		&\p\Big( \Big| \sum_{z\in B_n} \Xi_{1,r}(z,n) - \E[\Xi_{1,r}(z,n)] \Big| \ge  u_n  |\Wc_n| \Big) \\
		&\lesssim r^d \exp\Big( - c_1 \frac{\epsilon_1^2 \log n}{ 1 + 1/(a_n n^4) + \epsilon_1 (\log n)^{1+d/(2\as) }/(a_n n)^{1/2} } \Big)\\
		&\lesssim  r^d \exp\Big(- c_1 \frac{\epsilon_1^2 \log n}{1+o(1+\epsilon_1) } \Big).
\end{align*}
In particular, choosing $\epsilon_1$ sufficiently large (depending on $c_1$), implies the statement in (i).

For the statement in (i'), we rely additionally on Lemma~\ref{Lemma5} and Lemma~\ref{Lemma6}. We derive once more an upper bound which is valid uniformly in $1\le j \le S$. We use the decomposition 
\begin{align}\label{E:ExpIneqPoisson12}
	\sum_{z\in B_{n,j}} \sigma_{n,z}^2 &= \sum_{ \substack{ z\in B_{n,j}, \\   Q_z \not\subseteq \Wc_n}} \sigma_{n,z}^2  + \sum_{\substack{z\in B_{n,j}\\ Q_z \subseteq \Wc_n}} \sigma_{n,z}^2 .
\end{align}	
For the $\sigma^2_{n,z}$ in first sum in \eqref{E:ExpIneqPoisson12}, we rely on the upper bound in \eqref{E:ExpIneqPoisson13}.

The $\sigma^2_{n,z}$ in the second sum in \eqref{E:ExpIneqPoisson12} are treated with Lemma~\ref{Lemma5} and Lemma~\ref{Lemma6}:
\begin{align}
		\sigma^2_{n,z} &\le \E\Big[ \Xi_{1,r}(z,n)^2 \Big]   \nonumber\\
		&\le c_1[ \p(  \xi_r(0,\cP \cup \{0,Y\} ) \in J ) + \p(  \xi_r(0,\cP \cup \{0\} ) \in J ) ] \nonumber \\
		&\le c_1 [ \p(  \xi(0,\cP \cup \{0,Y\} ) \in J ) + \p(  \xi(0,\cP \cup \{0\} ) \in J ) ] + c_2 n^{-2} \nonumber \\
		&\lesssim a_n + n^{-2}, \label{E:ExpIneqPoisson14}
\end{align}
for $c_1,c_2\in\R_+$ which do not depend on $n$, $r$ and $I$; the last inequality follows because both probabilities are of order $|J| \lesssim a_n$.

Combining the estimates in \eqref{E:ExpIneqPoisson13} and \eqref{E:ExpIneqPoisson14}, we obtain for the sum in \eqref{E:ExpIneqPoisson12}
\begin{align*}
	&\sum_{ \substack{ z\in B_{n,j}, \\   Q_z \not\subseteq \Wc_n}} \sigma_{n,z}^2  + \sum_{\substack{z\in B_{n,j}\\ Q_z \subseteq \Wc_n}} \sigma_{n,z}^2 \\
	&\lesssim \frac{n^{(d-1)/d} r}{r^d} \log n \ (a_n+n^{-4})+ \frac{n}{r^d} (a_n + n^{-2}) \lesssim \frac{n}{r^d} a_n.
\end{align*} 
Setting this time
$$
	u_n = \epsilon_2 ( (\log n)^{1+d/\as} a_n / n)^{1/2},
$$
for $\epsilon_2\in\R_+$, yields that  \eqref{E:ExpIneqPoisson11} is of order $r^d \exp( - c_2 \epsilon_2^2 \log n / (1 + o(\epsilon_2) ))$ for a certain $c_2\in\R_+$. Choosing $\epsilon_2$  sufficiently large shows the statement (i').

We continue with \eqref{E:ExpIneqPoisson11} for the statement in (ii). In this case $J = (-\infty,x]$ for some $x\in I$. Then clearly
$
	\sigma^2_{n,z} \lesssim 1.
$ Consequently, setting 
\[
	u_n = \epsilon_3 (\log n)^{(1+d/\as)/2} n^{-1/2} ,
	\]
 we obtain that  \eqref{E:ExpIneqPoisson11} is of order $r^d \exp( - c_3 \epsilon_3^2 \log n / (1 + o(\epsilon_3) ))$ for a certain $c_3\in\R_+$. Setting $\epsilon_3$ sufficiently large, we obtain the statement in (ii). 
\end{proof}

We continue with the main result necessary for the Bahadur representation.
\begin{proposition}\label{P:UniformConv}
%Let $r=(c^*\log n)^{1/\as}$, where $c^* \ge 4/c_{stab}$.
 Let $z\in\R$ and $I_n = (z-a_n, z + a_n)$ for $n\in\N$, where the sequence $(a_n)_n$ is positive and bounded. Set for $x\in\R$
$$
	G_n (x) = \{ \wh F_n(x) - \wh F_n(z) \} - \{ F(x) - F(z) \}.
$$
\begin{itemize}
	\item [(i)] Let $(a_n)_n$ satisfy $(\log n)^{2+d/\as} = o(a_n n)$. Let $F$ be differentiable in some neighborhood of $z$ with $f$ being bounded in this neighborhood. There is a $c_1 \in \R_+$, which does not depend on the choice of $(a_n)_n$, such that for $\gamma=2$
\begin{align}\label{E:UniformConv0}
	\p \Big( \sup_{x\in I_n} |G_n (x)| \ge  c_1 \Big( \frac{(\log n)^{\gamma+d/\as} a_n }{n} \Big)^{1/2} \Big) = O(n^{-2}).
\end{align}\\[-8pt]
\item [(i')]
Let $(a_n)_n$ satisfy instead $(\log n)^{1+d/\as} = o(a_n n)$.
Additionally to (i) assume that $F^e$ from \eqref{E:DistributionExt} is differentiable in some neighborhood of $z$ and admits a bounded density $f^e$ in this neighborhood. There is a $c_1\in\R_+$, which does not depend on the choice of $(a_n)_n$, such that we have the result in \eqref{E:UniformConv0} for $\gamma=1$.\\[-4pt]
	\item [(ii)] Let $(a_n)_n$ satisfy $(\log n)^{2+d/\as} = o(a_n n)$. Let $I=(a,b)\subseteq\R$. Let $F$ be differentiable on $I$ with $\sup_{x\in I} f(x)  < \infty$. There is a $c_2 \in \R_+$, which does not depend on the choice of $(a_n)_n$, such that for $\gamma=2$
\begin{align}\begin{split}\label{E:UniformConv0B}
	&\p \Big( \sup_{\substack{x,y\in I\\ |x-y|\le a_n}} | G_n(x)-G_n(y)| \ge  c_2 \Big( \frac{(\log n)^{\gamma+d/\as} a_n }{n} \Big)^{1/2} \Big)\\
	& = O(n^{-2}).
\end{split}
\end{align}
(Note that the difference $G_n(x)-G_n(y)$ does not depend on the point $z$.)\\[-4pt]
\item [(ii')]
Let $(a_n)_n$ satisfy instead $(\log n)^{1+d/\as} = o(a_n n)$.
Assume additionally to (ii) that $F^e$ from \eqref{E:DistributionExt} admits a density $f^e$ on $I$ such that $\sup_{x\in I} f^e(x) <\infty$. There is a $c_2\in\R_+$, which does not depend on $(a_n)_n$, such that we have the result in \eqref{E:UniformConv0B} for $\gamma=1$.\\[-4pt]
	\item [(iii)] Let $I=(a,b)\subseteq\R$. Let $F$ be differentiable on $I$ with $\sup_{x\in I} f(x)  < \infty$. There is a $c_3\in\R_+$ such that
	$$
		\p \Big( \sup_{x\in I} |\wh F_n(x) - F(x) | \ge c_3 \Big( \frac{(\log n)^{1+d/\as}}{n} \Big)^{1/2} \Big)  = O(n^{-5/2}).
	$$
\end{itemize}
\end{proposition}
\begin{proof}
We begin with (i), (ii) is then an immediate consequence. In the last part of the proof, we show (iii) separately. The variants (i') and (ii') follow in a similar spirit as the original versions (i) and (ii) making use of Proposition~\ref{P:ExpIneqPoisson} (i') instead of part (i) there, so we do not give a separate proof for these.

 Note that given the condition on the positive and bounded sequence $(a_n)_n$ we have in each scenario that $1/n \lesssim a_n \lesssim 1$.

First we discretize the supremum as follows: Let $(b_n)_n\subseteq \N$ be a sequence of integers such that $b_n \asymp n$. Let $\eta_{k,n} = z + k a_n / b_n$ for $k\in\{-b_n,\ldots,b_n\}$ be a grid of $(z - a_n, z + a_n)$. If $x\in [\eta_{k,n},\eta_{k+1,n}]$, then using that $F$ is an increasing function in a neighborhood of $z$, we arrive at (given $n$ is sufficiently large)
\begin{align*}
	G_n (\eta_{k,n}) - \alpha_{k,n} \le G_n (x) \le G_n (\eta_{k+1,n}) + \alpha_{k,n},
\end{align*}
where $\alpha_{k,n} = F(\eta_{k+1,n} ) - F(\eta_{k,n}) \lesssim a_n / b_n$ uniformly for $x\in I_n$. Hence,
\begin{align}\label{E:UnifConv0}
	\sup_{x\in I_n} |G_n (x)| \le \max_{k: |k|\le b_n} |G_n (\eta_{k,n})| +  \max_{k: |k|\le b_n} \alpha_{k,n},
\end{align}
where the cardinality of the index set $\{k: |k|\le b_n \}$ is of order $n$ and the second maximum on the right-hand side of \eqref{E:UnifConv0} is deterministic and satisfies
\begin{align}
\begin{split}\label{E:Remainder}
	&\frac{1}{n^2} \lesssim \max_{k: |k|\le b_n} \alpha_{k,n} \lesssim \frac{a_n}{b_n}\\
	& \lesssim \frac{1}{n} \lesssim \frac{ (\log n)^{1+d/(2\as)} }{n} \lesssim \Big(\frac{ (\log n)^{2+d/\as} a_n }{n}\Big)^{1/2} .
\end{split}\end{align}
In particular, the second maximum is negligible in the further analysis.

Next, we study the first maximum on the right-hand side in \eqref{E:UnifConv0}. For this purpose, we split $|G_n (x)|$ as follows. Denote by $J_x$ the interval $(z,x]$ or $(x,z]$ depending on the position of $x$ relative to $z$. Then for each $x$
\begin{align}
	|G_n (x)| &= | \wh F_n(x) - \wh F_n(z) - F(x) + F(z) | \nonumber  \\
	&\le \Big | \big( \wh F_n(x) - \wh F_n(z) \big) - \big( \wh F_{n,r}(x) - \wh F_{n,r}(z) \big)	\Big | \label{E:UnifConv1} \\
	\begin{split}\label{E:UnifConv2}
	&\quad +  \frac{|\Wc_n|}{\Mc_n} \Big | \wh H_{n,r}(x)-\wh H_{n,r}(z) - \E[ \wh H_{n,r}(x)-\wh H_{n,r}(z) ] \Big|
	\end{split}\\
	\begin{split}\label{E:UnifConv3}
	&\quad + \Big| \frac{1}{\Mc_n} \E\Big[\sum_{y\in \cPc_n}  \1{\xi_r(y,\cP_n)\in J_x} \Big] \\
	&\quad\qquad\qquad\qquad - \E\Big[ \1{\xi_r(0,\cP\cup\{0\})\in J_x} \Big] \Big|
	\end{split}\\
	\begin{split}\label{E:UnifConv4}
	&\quad + \Big | \E\Big[ \1{\xi_r(0,\cP\cup\{0\}) \in J_x} -  \1{\xi(0,\cP\cup\{0\}) \in J_x} \Big] \Big |.
	\end{split}
\end{align}
(Note that the second term in \eqref{E:UnifConv2} and the first term in \eqref{E:UnifConv3} can be infinite. However, this is not a hindrance in the following because it only occurs with a probability which vanishes exponentially.)

For each $x$, the term in \eqref{E:UnifConv4} is at most
\[
	\E[ \1{\xi_r(0,\cP\cup\{0\}) \neq \xi(0,\cP\cup\{0\}) } ] \le C_{stab} \exp( -c_{stab} r^{\as} ) \lesssim n^{-4}
\] 
and is negligible in the first maximum in \eqref{E:UnifConv0}. 

Moreover, by Lemma~\ref{Lemma1} the term in \eqref{E:UnifConv1} satisfies for each $\epsilon >0$
\[
	\p\Big( \sup_{x\in\R} \Big | \big( \wh F_n(x) - \wh F_n(z) \big) - \big( \wh F_{n,r}(x) - \wh F_{n,r}(z) \big)	\Big | > \epsilon n^{-2} \Big) = O(n^{-2})
\]
and is as well negligible in the maximum in \eqref{E:UnifConv0}.

The term in \eqref{E:UnifConv3} admits an upper bound which is uniform in $x \in I_n$ and negligible, so taking the maximum does not have an effect. To see this, we rely on the Slyvniak-Mecke formula:
\begin{align*}
	 \E\Big[\sum_{y\in \cPc_n}  \1{\xi_r(y,\cP_n)\in J_x} \Big]  &= \int_{\Wc_n} \E\Big[ \1{	\xi_r(y, \cPc_n \cup \{ y\} \cup \cPd_n) \in J_x } \Big] \diff y \\
	 &= |\Wc_n| \ \p( \xi_r(0,\cP\cup\{0\}) \in J_x),
\end{align*}
where as before for all $x\in I_n$
\[
	 \p( \xi_r(0,\cP\cup\{0\}) \in J_x) \le  \p( \xi(0,\cP\cup\{0\}) \in J_x) + O( \exp(-c_{stab} r^{\as}) ) \lesssim a_n + n^{-4}.
\]
So there are $c,c'\in\R_+$, which do not depend on $(a_n)_n$, such that for $\epsilon>0$ sufficiently large
\begin{align*}
	&\p\Big( \sup_{x\in I_n} \Big| \frac{1}{\Mc_n} \E\Big[\sum_{y\in \cPc_n}  \1{\xi_r(y,\cP_n)\in J_x} \Big] \\
	&\qquad\qquad - \E\Big[ \1{\xi_r(0,\cP\cup\{0\})\in J_x} \Big] \Big| \ge \frac{\epsilon a_n (\log n)^{1/2}}{ n^{1/2} } \Big ) \\
	&\le \p\Big(  \sup_{x\in I_n} | |\Wc_n| - \Mc_n | \p( \xi_r(0,\cP\cup\{0\}) \in J_x) \ge \frac{\epsilon a_n \Mc_n (\log n)^{1/2}}{ n^{1/2}} \Big) \\
	&\le \p\Big( | |\Wc_n| - \Mc_n | \ge c \frac{ \epsilon a_n |\Wc_n| (\log n)^{1/2}}{n^{1/2}(a_n + n^{-4}) }  \Big) + \p\Big( \frac{ \Mc_n }{ |\Wc_n | } < \frac{1}{2} \Big) \\
	&\le  \p\Big( | |\Wc_n| - \Mc_n | \ge c' \frac{ \epsilon  n^{1/2} (\log n)^{1/2} }{1 + o(1) }  \Big) + 2\exp\Big( - \frac{|\Wc_n|}{12}		\Big) = O( n^{-2} ).
\end{align*}
The last estimate is a consequence of Lemma~\ref{Lemma0} given $\epsilon>0$ is sufficiently large. Since $(a_n)_n$ is bounded and
\[
	\frac{a_n (\log n)^{1/2}}{n^{1/2}} \lesssim \Big(\frac{(\log n)^{1+d/\as} a_n}{n} \Big)^{1/2}
\]
this term is negligible as well.

We come to the dominating term \eqref{E:UnifConv2}, which is treated in detail in Proposition~\ref{P:ExpIneqPoisson} (i). For each $\tau>1$ there is constant $C_{\tau}\in\R_+$ such that for all positive bounded sequences $(a_n)_n$ which satisfy $(\log n)^{2+d/\as} = o(a_n n)$,
\begin{align*}
	&\p\Big( \max_{k: |k|\le b_n} \big| \wh H_{n,r}(\eta_{k,n}) - \wh H_{n,r}(z) - \E[\wh H_{n,r}(\eta_{k,n}) - \wh H_{n,r}(z)] \big| \\
	&\ge C_{\tau} \Big( \frac{(\log n)^{2+d/\as} a_n }{n} \Big)^{1/2} \Big) = O( b_n n^{-\tau} ).
\end{align*}
Moreover, the event $\{|\Wc_n| / \Mc_n > 2\}$ is negligible because we have by Lemma~\ref{Lemma0} $\p(\Mc_n / |\Wc_n|  < 1/2) \le 2 \exp( - |\Wc_n| / 12)$. Consequently for the term in \eqref{E:UnifConv2} there is a $C_3\in\R_+$ such that
\begin{align*}
		&\p\Big( \max_{k: |k|\le b_n} \frac{|\Wc_n|}{\Mc_n} \Big | \wh H_{n,r}(\eta_{k,n})-\wh H_{n,r}(z) - \E[ \wh H_{n,r}(\eta_{k,n})-\wh H_{n,r}(z) ] \Big| \\
		&\qquad\qquad\qquad \ge C_{3} \Big( \frac{(\log n)^{2+d/\as} a_n }{n} \Big)^{1/2} \Big) = O(n^{-2}) .
\end{align*}
Finally, we come back to the first maximum on the right-hand side of \eqref{E:UnifConv0}. Using the preceding estimates, there is a $c\in\R_+$ such that
\begin{align*}
	 &\p\Big( \max_{k: |k|\le b_n} |G_n (\eta_{k,n})| \ge c \Big( \frac{(\log n)^{2+d/\as} a_n }{n} \Big)^{1/2}\Big) = O(n^{-2})
\end{align*}
for each $(a_n)_n$, which satisfies $(\log n)^{2+d/\as} = o(a_n n)$.
This shows (i).

For (ii), let once more $b_n \asymp n$. We cover $I$ using another sequence $(\wt\eta_{k,n})_{k=0}^{K+1}$, where this time $K = \ceil{(b-a)b_n/a_n}$ and $\wt\eta_{k,n} = a + k a_n/b_n$ for $k\in\Z$. So, $I \subseteq \cup_{k=0}^K [\wt\eta_{k,n},\wt\eta_{k+1,n}]$. Let $x\in [\wt\eta_{k,n},\wt\eta_{k+1,n}]$ (for one $1\le k \le K-1$) and $y\in I$ such that $|x-y|\le a_n/b_n$, then
\begin{align*}
	&\wh F_n(x) - \wh F_n(\wt\eta_{k+2,n}) - F(x) + F(\wt\eta_{k+2,n}) \\
	&\qquad\qquad + (F(\wt\eta_{k-1,n})-F(\wt\eta_{k+2,n})) \\
	&\le G_n(x)-G_n(y) \\
	&\le \wh F_n(x) - \wh F_n(\wt\eta_{k-1,n}) - F(x) + F(\wt\eta_{k-1,n}) \\
	&\qquad\qquad + (F(\wt\eta_{k+2,n})-F(\wt\eta_{k-1,n})).
\end{align*}
Similar upper bounds are valid if $x\in [\wt\eta_{0,n},\wt\eta_{1,n}]$ or $x\in [\wt\eta_{K,n},\wt\eta_{K+1,n}]$.
Consequently, we find with similar considerations as in the first part that
\begin{align*}
	&\sup_{x,y\in I} | G_n(x)-G_n(y)| \\
	&\le \max_{k\in \{0,\ldots,K+1\}} \sup_{x\in [\wt\eta_{k-1,n}, \wt\eta_{k+2,n}] } | \wh F_n(x) - \wh F_n (\wt\eta_{k,n}) - F(x) + F(\wt\eta_{k,n}) | + O(a_n/b_n) \\
	&\le \max_{k\in \{0,\ldots, K+1\} }  \Big\{ | \wh F_n(\wt\eta_{k+2,n}) - \wh F_n (\wt\eta_{k,n}) - F(\wt\eta_{k+2,n}) + F(\wt\eta_{k,n}) |, \\
	&\qquad\qquad | \wh F_n(\wt\eta_{k,n}) - \wh F_n (\wt\eta_{k-1,n}) - F(\wt\eta_{k,n}) + F(\wt\eta_{k-1,n}) | \Big\} + O(a_n/b_n) ,
\end{align*}
where the remainder is deterministic, uniform in $k$ and in $n$ and satisfies similar inequalities as in \eqref{E:Remainder}.

Since each $[\wt\eta_{k-1,n},\wt\eta_{k+2,n}]$ is of order $a_n/b_n \lesssim a_n$, we can apply the estimate from part (i) to the first term, using a decomposition as in \eqref{E:UnifConv1} to \eqref{E:UnifConv4}. Note that the results here are indeed uniform in the position $\wt\eta_{k,n}$: The upper bounds on the terms corresponding to those in \eqref{E:UnifConv1}, \eqref{E:UnifConv3} and \eqref{E:UnifConv4} are clearly uniform in the $\wt\eta_{k,n}$. Regarding the dominating term which corresponds to \eqref{E:UnifConv2}, we can once more rely on Proposition~\ref{P:ExpIneqPoisson} (i) because the number of intervals $[\wh\eta_{k,n},\wh\eta_{k+2,n}], [\wh\eta_{k-1,n},\wh\eta_{k,n}]$ is of order $b_n/a_n \lesssim n^{2}$. This shows (ii).

For (iii) we proceed very similarly. Define $u_n = n^{-1/2} (\log n)^{1/2+d/(2\as) }$ and $\eta'_{k,n} = a + k u_n$ for $k\in \N_0$. We can cover $I$ with $\cup_{k=0}^K [\eta'_{k,n},\eta'_{k+1,n}]$ for $K\lesssim u_n^{-1}$. Using that $\sup_{x\in I} F'(x) < \infty$, we see that
$$
	\sup_{x\in I} |\wh F_n(x) - F(x) | \le \max_{k\in \{1,\ldots,K+2\} } |\wh F_n(\eta'_{k-1,n}) - F(\eta'_{k-1,n}) | + O( u_n ). 
$$
Then, we rely on the decomposition
\begin{align}
	&|\wh F_n(x) - F(x)| \nonumber\\
	&\le |\wh F_n(x) - \wh F_{n,r}(x)| \label{E:UnifConv5} \\
	\begin{split}\label{E:UnifConv6}
	&\quad + \frac{|\Wc_n|}{\Mc_n} \Big | \wh H_{n,r}(x)  - \E\Big[\wh H_{n,r}(x) \Big]	\Big|
	\end{split}\\
	\begin{split}\label{E:UnifConv7}
	&\quad + \Big| \frac{1}{\Mc_n} \E\Big[\sum_{y\in \cPc_n}  \1{\xi_r(y,\cP_n)\le x} \Big] - \E\Big[ \1{\xi_r(0,\cP\cup\{0\})\le x} \Big] \Big|
	\end{split}\\
	&\quad + | \E[ \1{\xi_r(0,\cP\cup \{0\}) \le x} - \1{\xi(0,\cP\cup \{0\}) \le x} ] | \label{E:UnifConv8}
\end{align}
for each $x\in \R$. (Again the term in \eqref{E:UnifConv6} and the term in \eqref{E:UnifConv7} can be infinite.)

The term in \eqref{E:UnifConv8} corresponds to \eqref{E:UnifConv4} and is $O(n^{-4})$ uniformly in $x$. Similarly, the term in \eqref{E:UnifConv5} corresponds to \eqref{E:UnifConv1} and satisfies again for each $\epsilon>0$
\[
	\p\big( \sup_{x\in\R} \big| \wh F_n(x)  - \wh F_{n,r}(x) \big) | > \epsilon n^{-2} \big) = O(n^{-2}).
\]
The term in \eqref{E:UnifConv7} corresponds to \eqref{E:UnifConv3} and satisfies for $\epsilon>0$ sufficiently large
\begin{align*}
	&\p\Big( \sup_{x\in\R} \Big| \frac{1}{\Mc_n} \E\Big[\sum_{y\in \cPc_n}  \1{\xi_r(y,\cP_n)\le x} \Big] \\
	&\qquad - \E\Big[ \1{\xi_r(0,\cP\cup\{0\})\le x} \Big] \Big| \ge \frac{\epsilon (\log n)^{1/2} }{ n^{1/2} } \Big ) = O(n^{-2}).
\end{align*}
In particular, this term is negligible as well.

The dominating term in \eqref{E:UnifConv6} corresponds to \eqref{E:UnifConv2}. An application of Proposition~\ref{P:ExpIneqPoisson} (ii) and the fact that $K\lesssim u_n^{-1}$ yield the existence of a $C_3\in\R_+$
\begin{align*}
	&\p\Big( \max_{k\in \{1,\ldots,K+2\} }	\Big | \wh H_{n,r}(\eta'_{k-1,n})  - \E\Big[\wh H_{n,r}(\eta'_{k-1,n}) \Big]	\Big| \\
	&\qquad\qquad\qquad \ge C_3 \Big( \frac{(\log n)^{1+d/\as}}{n} \Big)^{1/2} \Big) = O( u_n^{-1} \ n^{-3} ) = O(n^{-5/2}).
\end{align*}
 This shows (iii) and completes the proof.
\end{proof}

\begin{proof}[Proof of Theorem~\ref{T:BahadurPoissonUnif}]
We show in detail part (i) of the theorem; part (ii) works in a very similar fashion.

\textit{Part (i).} Let $I\subseteq\R$ be a bounded interval. Given the assumptions of Proposition~\ref{P:UniformConv} (ii) and (iii) there is a constant $\wt c\in\R_+$ such that for any positive and bounded sequence $(a_n)_n$, which satisfies $(\log n)^{2+d/\as} = o(a_n n)$, we have both
\begin{align*}
	&\limsup_{n\to\infty} \frac{n^{1/2}}{a_n^{1/2} (\log n)^{1+d/(2\as)}} \sup_{\substack{x,y\in I,\\ |x-y|\le a_n }} | (\wh F_n(x) - \wh F_n(y)) - (F(x)-F(y)) | \le \wt c, \\
	&\limsup_{n\to\infty} \frac{ n^{1/2}}{ (\log n)^{1/2 + d/(2\as)}} \sup_{x\in I} |\wh F_n(x) - F(x) | \le \wt c
\end{align*}
with probability 1.

In the following, we apply these results to the interval $\fP = (\psi_{p_0}-\epsilon, \psi_{p_1}+\epsilon )$, where by assumption $\epsilon>0$. The requirements of Proposition~\ref{P:UniformConv} are satisfied. Indeed, we have by the fundamental theorem of calculus $f(x) = f(z) + \int_z^x f'(s) \diff s$ for all $x,z\in \fP$. Consequently, as $\int_{\fP} f(s) \diff s \le 1$ and $\sup_{x\in \fP} f'(x)<\infty$, we have $\sup_{x\in \fP} f(x) < \infty$.

We choose $a_n = \ol c n^{-1/2} (\log n)^{1/2 + d/(2\as)}$ for $\ol c = (\wt c + 1)/\inf_{x\in \fP} f(x)$.
We can make this choice because the constant $\wt c$ does not depend on $\ol c$. Moreover, we can assume w.l.o.g. that $a_n \le \epsilon$ for all $n\in\N$.

 Clearly, we have with probability 1
\begin{align*}
		&\inf_{p_0\le p \le p_1} (\wh F_n( \psi_p + a_n ) - p) \\
		&\ge  \inf_{p_0\le p \le p_1} (F(\psi_p+a_n)-p) - \sup_{x\in \fP} |\wh F_n(x) - F(x) | \\
		&\quad - \sup_{\substack{x,y\in \fP,\\ |x-y|\le a_n }} | (\wh F_n(x) - \wh F_n(y)) - (F(x)-F(y)) |.
\end{align*}
Consequently, with probability 1
\begin{align*}
		&\liminf_{n\to\infty} \frac{ n^{1/2}}{ (\log n)^{1/2 + d/(2\as)}}  \inf_{p_0\le p \le p_1} (\wh F_n( \psi_p + a_n ) - p) \\
			&\ge \liminf_{n\to\infty} \frac{ n^{1/2}}{ (\log n)^{1/2 + d/(2\as)}} \Big\{ a_n \inf_{x\in\fP} f(x) + O( a_n^2 ) \Big\} -  \wt c \ge 1.
\end{align*}
Similarly, with probability 1
$$
	\limsup_{n\to\infty} \frac{ n^{1/2}}{ (\log n)^{1/2 + d/(2\as)}}  \sup_{p_0\le p \le p_1} (\wh F_n( \psi_p - a_n ) - p) \le -1.
$$
This shows, $\{ \sup_{p_0\le p\le p_1} |\wh \psi_{p,n} - \psi_p| > a_n \text{ i.o.}\}$ has probability 0. I.e., eventually $\sup_{p_0\le p\le p_1} |\wh \psi_{p,n} - \psi_p| \le a_n$ with probability 1.

Moreover, $|\wh F_n(\wh \psi_{p,n})-p|\le 1/n$. Thus, we can apply Proposition~\ref{P:UniformConv} (ii)
\begin{align*}
	& \sup_{p_0 \le p \le p_1} | (p-F(\wh \psi_{p,n})) - (\wh F_n(\psi_p)-F(\psi_p)) | \\
	&= \sup_{p_0 \le p \le p_1} | (\wh F_n(\wh \psi_{p,n})-F(\wh \psi_{p,n})) - (\wh F_n(\psi_p)-F(\psi_p)) | + O(n^{-1})\\
	&=O_{a.s.}( n^{-1/2} a_n^{1/2}  (\log n)^{1+d/(2\as)}) = O_{a.s}(n^{-3/4} (\log n)^{5/4+3d/(4\as) } ).
\end{align*}

Using a Taylor expansion of order 2 at $\psi_p$ together with the requirements $\inf_{x\in\fP} f(\psi_p) > 0$ as well as $\sup_{x\in\fP} |f'(x)| < \infty$ entails the claim.

\textit{Part (ii).} Given the assumptions in this setting, by Proposition~\ref{P:UniformConv} (ii') and (iii) there is a constant $\wt c\in\R_+$ such that for any positive and bounded sequence $(a_n)_n$, which satisfies $(\log n)^{1+d/\as} = o(a_n n)$, we have both
\begin{align*}
	&\limsup_{n\to\infty} \frac{n^{1/2}}{a_n^{1/2} (\log n)^{1/2+d/(2\as)}} \sup_{\substack{x,y\in I,\\ |x-y|\le a_n }} | (\wh F_n(x) - \wh F_n(y)) - (F(x)-F(y)) | \le \wt c, \\
	&\limsup_{n\to\infty} \frac{ n^{1/2}}{ (\log n)^{1/2 + d/(2\as)}} \sup_{x\in I} |\wh F_n(x) - F(x) | \le \wt c
\end{align*}
with probability 1. Now, we can repeat all calculations with the analog sequence $(a_n)_n$ to see with the help of Proposition~\ref{P:UniformConv} (ii') that
\begin{align*}
	& \sup_{p_0 \le p \le p_1} | (p-F(\wh \psi_{p,n})) - (\wh F_n(\psi_p)-F(\psi_p)) | =O_{a.s}(n^{-3/4} (\log n)^{3/4(1+d/\as) } ).
\end{align*}
This shows part (ii) and completes the proof.
\end{proof}

\section{Proofs of the results in Section~\ref{Section_Applications}}
\label{Section_Supplement}

\subsection{An abstract law of the iterated logarithm}
In this section we give some abstract results which lay the groundwork for functional CLT and the LIL in Section~\ref{Section_Applications}. For this purpose and given the processes $\cP_n$ and $\cPc_n$, we rely on the following score functionals. Let $m\in\N$, $x_1,\ldots,x_m\in\R$ and $a_1,\ldots,a_m\in\R$. We define for a finite point cloud $P$ and $y\in P$ the scores
\begin{align}\label{E:FidiScores}
	\Psi(y,P) = \sum_{i=1}^m a_i (\1{\xi(y,P)\le x_i } - F(x_i) ),
\end{align}
where $F(x) = \p( \xi(0,\cP\cup\{0\}) \le x)$ is the distribution function of the score functional if a homogeneous Poisson process with unit intensity on $\R^d$ is used.

Using the Poisson processes from \eqref{D:PP1} to \eqref{D:PP3}, we consider the differences
\begin{align}
	\Delta(z,n) &= \sum_{y\in\cPc_n} \Psi(y,\cP_n) - \sum_{y\in\cPc_{n,z} } \Psi(y,\cP''_{n,z}) 	\label{E:MDS1} \\
	\begin{split}\label{E:MDS2}
	\Delta(z,\infty) &=  \sum_{y\in\cP \setminus Q_z} \sum_{i=1}^m a_i (\1{\xi(y,\cP)\le x_i } - \1{\xi(y,\cP''_z)\le x_i} )  \\
	&\quad  + \sum_{y\in\cP \cap Q_z } \Psi(y,\cP) - \sum_{y\in \cP''_z \cap Q_z  } \Psi(y,\cP''_{z})
	\end{split}
\end{align}
First, note that
\[
	( \E[ \Delta(z,n) | \cF_z ] : z\in B_n ) \text{ and } ( \E[ \Delta(z,\infty) | \cF_z ] : z\in B_n ) 
\]
are both martingale differences w.r.t.\ the Poisson filtration $(\cF_z: z\in\Z^d)$, which is based on the lexicographic ordering of $\Z^d$. In particular,
\[
	\sum_{y\in \cPc_n} \Psi(y,\cP_n) - \E[\sum_{y\in \cPc_n} \Psi(y,\cP_n)] = \sum_{z\in B_n} \E[ \Delta(z,n) | \cF_z].
\]

Note that the first sum in \eqref{E:MDS2} is well-defined because of the next Lemma~\ref{L:DeltaInfty}.
\begin{lemma}\label{L:DeltaInfty}
Let $z\in \R^d$. %and let $r=(c^*\log n)^{1/\as}$, where $c^* \ge 4/c_{stab}$.
Then there are constants $c_1,c_2\in\R_+$ such that
\begin{align}
	\p\Big( \exists y \in \cP \cap B(z,k)^c : R(y,\cP) \ge \|y-z\| - \frac{\sqrt{d}}{2} \Big) \le c_1 \exp( - c_2 k^{\as} ),\label{E:DeltaInfty1}
\end{align}
for each $k\in\R_+$. Hence, $\p( \#\{y \in \cP: R(y,\cP) \ge \|y-z\| - \sqrt{d}/2 \}<\infty ) = 1$.

Moreover,
\begin{align}
	\p(\exists N\in\N |\forall n\ge N, \forall y\in \cPc_n: R(y,\cP) \le r ) = 1.	\label{E:DeltaInfty2}
\end{align}

In particular, $\lim_{n\to\infty} \Delta(z,n) = \Delta(z,\infty)$ $a.s.$ and the definition of $\Delta(z,\infty)$ is meaningful for each $z\in\R^d$. Moreover, $\fD_\infty(x)$ from \eqref{E:AddOneCost} is a proper random variable for each $x\in\R$.
\end{lemma}
\begin{proof}
We begin with \eqref{E:DeltaInfty1}. Set $S_j(z) = B(z,j) \setminus B(z,j-1)$. Then we have for $A_k = \{ \exists y \in \cP \cap B(z,k)^c : R(y,\cP) \ge \|y-z\| - \sqrt{d}/2 \}$
\begin{align}
	\p( A_k  ) &= \sum_{j=1}^\infty \p( \exists y \in \cP \cap B(z,k)^c \cap S_j(z) : R(y,\cP) \ge \|y-z\| - \sqrt{d}/2 ) \nonumber \\
	&\le \sum_{j=1}^\infty \E\Big[ \sum_{ y \in \cP \cap B(z,k)^c \cap S_j(z) } \1{ R(y,\cP) \ge \|y-z\| - \sqrt{d}/2 } \Big] \nonumber \\
	& \le \sum_{j=1}^\infty | B(z,k)^c \cap S_j(z) | \ \p( R(0,\cP\cup \{0\}) \ge (j-1)-\sqrt{d}/2 ) \nonumber \\
	&\lesssim \sum_{j=k+1}^\infty  j^{d-1} \ \exp( - c_{stab} j^{\as} ) \nonumber \\
	&= \sum_{j=1}^\infty (j+k)^{d-1} \ \exp( - c_{stab} (j+k)^{\as} )  \nonumber \\
	&\lesssim \exp( - a k^{\as} ) \label{E:DeltaInfty5},
\end{align}
for some $a\in\R_+$ and
where we use $j^{\as} + k^{\as} \le 2 (j\vee k)^{\as} \le 2 (j+k)^{\as}$  for the last inequality. We infer from \eqref{E:DeltaInfty5} that $\p( A_k \text{ occurs i.o.}) = 0$.

Regarding \eqref{E:DeltaInfty2}, we see
\begin{align*}
	&\sum_{n=1}^\infty \p( \exists y \in \cPc_n : R(y,\cP) \ge r ) \le \sum_{n=1}^\infty \E\Big[ \sum_{y\in \cPc_n} \1{R(y,\cP) \ge r } \Big] \\
	&\lesssim \sum_{n=1}^\infty n \ \p( R(0,\cP\cup\{0\} ) \ge r ) < \infty.
\end{align*}
Hence, $\p( \{\exists y \in \cPc_n : R(y,\cP) \ge r\} \text{ occurs i.o.} ) = 0$. This shows \eqref{E:DeltaInfty2}.

We conclude with the amendment of the lemma regarding the random variable $\Delta(z,\infty)$.
Given $z\in\R^d$, there is an $n_0\in\N$ such that $Q_z$ is contained in $\Wc_n$ for all $n\ge n_0$. Then for $n\ge n_0$, we decompose $\Delta(z,n)$ as follows.
\begin{align}
	\Delta(z,n) &= \sum_{y\in \cPc_n \setminus Q_z} \sum_{i=1}^m a_i( \1{ \xi(y,\cP_n)\le x_i} - \1{ \xi(y,\cP''_{n,z})\le x_i} ) \label{E:DeltaInfty3} \\
	&\quad + \sum_{y\in \cP \cap Q_z} \Psi(y,\cP_n) - \sum_{y\in \cP''_{z} \cap Q_z} \Psi(y,\cP''_{n,z} ). \label{E:DeltaInfty4}
\end{align}
The terms in \eqref{E:DeltaInfty4} converge to $\sum_{y\in \cP \cap Q_z} \Psi(y,\cP) - \sum_{y\in \cP''_z \cap Q_z} \Psi(y,\cP''_{z})$ $a.s.$ 

It remains to show that there is an $n_1\in\N$, $n_1\ge n_0$, which depends on $\omega\in\Omega$, such that the sum in \eqref{E:DeltaInfty3} is constant for all $n\ge n_1$. This follows however from \eqref{E:DeltaInfty1} and \eqref{E:DeltaInfty2}. To see this, note that the distance between a point $y$ outside $Q_z$ and the boundary $\partial Q_z$ is at least $\|y-z\| - \sqrt{d}/2$. Moreover, we have
\begin{align*}
	&\{ \exists k> 0, \exists N\in \N | \forall n \ge N, \forall y\in (\cPc_n \setminus Q_z) \cap B(z,k)^c: \\
	&\qquad\qquad\qquad\qquad\qquad\qquad R(y,\cP) \le r \wedge \|y-z\| - \sqrt{d}/2 \} \\
	&\supseteq \{ \exists k> 0, | \forall y\in \cP \cap B(z,k)^c: R(y,\cP) \le \|y-z\| - \sqrt{d}/2 \}  \\
	&\qquad \cap \{ \exists N\in \N | \forall n \ge N, \forall y\in \cPc_n \setminus Q_z: R(y,\cP) \le r \} .
\end{align*}
Both events on the right-hand side have probability 1. This shows the amendment.
\end{proof}

The next proposition contains an important approximation result.
\begin{proposition}\label{P:Approximation}
Let $p>2d$ let $r=(c^*\log n)^{1/\as}$, where $c^* \ge 4/c_{stab}$. 
Either assume the additional stabilization condition from \eqref{E:Stabilization3} or let additionally $c^* > (p+1+1/d)/c_{stab}$. Then we have for the martingale differences induced by \eqref{E:MDS1} and \eqref{E:MDS2}
\begin{align}
	&\sup_{n\in \N} \sup_{z\in B_n} \E[ |\Delta(z,n)|^p] + \E[ |\Delta(z,\infty)|^p ] < \infty, \label{E:ApproximationMain0}\\
	&\sum_{z\in B_n}  \E[ \Delta(z,n) - \Delta(z,\infty)  | \cF_z] = o_{a.s.} ( r^{1/2} (\log n)^{1+\delta} n^{ (d-1)/(2d) + 1/p} ) \label{E:ApproximationMain1}
\end{align}
for each $\delta>0$.
In particular, if $p>2d$, the right-hand side is $o_{a.s}(n^{1/2})$.
\end{proposition}
\begin{proof}
The calculations are split in four steps.

\textit{Step 1.}
Set $X_{n,z} =  \E[ \Delta(z,n) - \Delta(z,\infty)  | \cF_z]$ for $z\in B_n$ and $n\in\N$. Using the Burkholder and the Minkowski inequality, there is a $c_{1,p}\in\R_+$ such that
\begin{align}
	\Big\| \sum_{z\in B_n} X_{n,z} \Big \|_p \le  c_{1,p} \Big\| \Big(\sum_{z\in B_n} X_{n,z}^2 \Big)^{1/2} \Big \|_p   \le c_{1,p} \Big( \sum_{z\in B_n} \| X_{n,z}^2 \|_{p/2} \Big)^{1/2}.\label{E:Approximation0}
\end{align}
In order to give an upper bound on \eqref{E:Approximation0}, we derive suitable upper bounds on the $p$-norm of $\Delta(z,n)-\Delta(z,\infty)$ depending of the relative position of $z$ within the windows $W_n$. 

Throughout the rest of the proof, let $\wt p, \wt q$ be H{\" o}lder conjugate and not depending on $n$. If the extended stabilization condition from \eqref{E:Stabilization3} is in force, then set $\wt p = \wt q = 2$. Otherwise, choose $\wt q$ close enough to 1 such that
\[
	c^* \ge \frac{ \wt q (p + 1/d)+1}{c_{stab}}
\]
holds (this is possible because of the additional requirement on $c^*$ in this case). In particular, we have in this case
\begin{align}\label{E:Approximation3a}
	n^{p + 1/ \wt q} \ \p( R(0,\cP\cup\{0\}) \ge r )^{1/ \wt q} \lesssim n^{- (c_{stab} c^* / \wt q - p - 1/ \wt q ) } \le n^{-1/d}
\end{align}
what we will use below in Step 2 and 3.

We show in Step 2 that $\E[ | \Delta(z,n) - \Delta(z,\infty) |^p ] \le c_{3} n^{-1/d}$ if $Q_z \subseteq \Wc_n$ for a constant $c_3\in\R_+$ which does neither depend on $z$ nor on $n$.

We show in Step 3 that $\E[ | \Delta(z,n) - \Delta(z,\infty) |^p ] \le c_4$ for all $z\in B_n$ for a constant $c_4\in\R_+$ which does neither depend on $z$ nor on $n$.

\textit{Step 2.}
If $Q_z \subseteq \Wc_n$, then we use \eqref{E:MDS2} as well as \eqref{E:DeltaInfty3} and \eqref{E:DeltaInfty4} to obtain
\begin{align}
	&\E[ | \Delta(z,n) - \Delta(z,\infty) |^p ] \nonumber \\
	&\lesssim \E\Big[ \Big| \sum_{y\in \cPc_n \cap Q_z } \Psi(y,\cP_n) -  \sum_{y\in \cP \cap Q_z } \Psi(y,\cP)  \Big|^p \Big] \label{E:Approximation2} \\
	\begin{split}\label{E:Approximation3}
	&\qquad + \E\Big[ \Big|  \sum_{y\in \cPc_n \setminus Q_z} \sum_{i=1}^m a_i( \1{ \xi(y,\cP_n)\le x_i} - \1{ \xi(y,\cP''_{n,z})\le x_i} ) \\
	&\qquad\qquad - \sum_{y\in\cP \setminus Q_z} \sum_{i=1}^m a_i (\1{\xi(y,\cP)\le x_i } - \1{\xi(y,\cP''_z)\le x_i} ) \Big|^p \Big].
	\end{split}
\end{align}
Since $Q_z\subseteq \Wc_n$, the moment in \eqref{E:Approximation2} is
\begin{align}
	&\E\Big[	\Big| \sum_{y\in \cP\cap Q_z} \Psi(y,\cP_n) - \Psi(y,\cP) \Big|^p \Big] \nonumber \\
	 &\lesssim \E\Big[ \Big| \sum_{y\in \cP\cap Q_z} \1{ \xi(y,\cP_n) \neq \xi(y,\cP)} \Big|^p \Big]\nonumber \\
	 &\le \| \cP( Q_z )^p \|_{\wt p} \ \p( \exists y \in  \cP\cap Q_z : R(y,\cP)\ge r )^{1/\wt q} \nonumber \\
	 &\lesssim \p( R(0,\cP\cup\{0\} ) \ge r)^{1/\wt q} \lesssim \exp( - c_{stab} r^{\as} / \wt q ). \nonumber
\end{align}

We continue with the expectation in \eqref{E:Approximation3}.
Since $p$ and $m$ are fixed, it is sufficient to consider the following moment in order to show that \eqref{E:Approximation3} is small. Consider for $x\in\R$
\begin{align}
	&\E\Big[ \Big|  \sum_{y\in \cPc_n \setminus Q_z} ( \1{ \xi(y,\cP_n)\le x} - \1{ \xi(y,\cP''_{n,z})\le x} ) \nonumber \\
	&\qquad\qquad - \sum_{y\in\cP \setminus Q_z} (\1{\xi(y,\cP)\le x } - \1{\xi(y,\cP''_z)\le x} ) \Big|^p \Big] \nonumber\\
	&\lesssim \E\Big[ \Big|  \sum_{ \substack{ y\in \cPc_n \setminus Q_z, \\ \|y-z\| \le r  } }  \ \1{ \xi(y,\cP_n)\le x}  - \sum_{\substack{y\in\cP \setminus Q_z,\\ \|y-z\| \le r } } \1{\xi(y,\cP)\le x } \Big|^p \Big] \label{E:Approximation4} \\
	&\quad + \E\Big[ \Big|  \sum_{ \substack{ y\in \cPc_n \setminus Q_z, \\ \|y-z\| \ge r } }  \ \1{ \xi(y,\cP_n)\le x}  - \1{\xi(y,\cP''_{n,z})\le x } \Big|^p \Big] \label{E:Approximation5} \\
	&\quad + \E\Big[ \Big|  \sum_{ \substack{ y\in \cP \setminus Q_z, \\ \|y-z\| \ge r } }  \ \1{ \xi(y,\cP)\le x}  - \1{\xi(y,\cP''_{z})\le x } \Big|^p \Big]. \label{E:Approximation6}
\end{align}
We have to treat each term in \eqref{E:Approximation4} to \eqref{E:Approximation6} separately.

The sum in \eqref{E:Approximation4}, which treats the local scores, is at most
\begin{align}
	&\E\Big[ \Big|  \sum_{ \substack{ y\in \cP \setminus Q_z, \\ \|y-z\| \le r } }  \ \1{ \xi(y,\cP_n)\le x}  -  \1{\xi(y,\cP)\le x } \Big|^p \Big] \nonumber \\
	&\le \E\Big[ \cP( B(z, r))^{p}  \1{ \exists y\in \cP \setminus Q_z, \|y-z\| \le r :  \xi(y,\cP_n) \neq \xi(y,\cP)} \Big]  \nonumber \\
	&\le \| \cP( B(z, r))^{p} \|_{\wt p} \ \E\Big[ \sum_{ \substack{ y\in \cP \setminus Q_z, \\ \|y-z\| \le r } } \1{ R(y,\cP) \ge r } \Big] ^{1/\wt q} \nonumber \\
	&\lesssim r^{dp + d/\wt q} \ \p( R(0,\cP\cup \{0\} ) \ge r )^{ 1/\wt q}. \nonumber%\label{E:Approximation7}
\end{align}

We continue with the second non-local sum in \eqref{E:Approximation6} which is at most
\begin{align}
		&\E\Big[ \Big|  \sum_{ \substack{ y\in \cP \setminus Q_z, \\ \|y-z\| \ge r } }  \ \1{ \xi(y,\cP) \neq \xi(y,\cP''_{z}) } \Big|^p \Big] \nonumber \\
	&\le \E\Big[ \Big|  \sum_{ \substack{ y\in \cP \setminus Q_z, \\ \|y-z\| \ge r } }  \1{R(y,\cP) \ge \|y-z\| - \sqrt{d}/2 } 	\Big|^p \Big] \nonumber \\
	&= \sum_{j_1,\ldots,j_p=1}^\infty \E \Bigg[ \prod_{i=1}^p \sum_{y\in S_{j_i+r }(z) \cap \cP }	\1{R(y,\cP) \ge \|y-z\| - \sqrt{d}/2 }  \Bigg], \label{E:Approximation9}
\end{align}
where $S_{j}(z) = B(z,j)\setminus B(z,j-1)$. It follows with elementary calculations that \eqref{E:Approximation9} is at most
\begin{align*}
	&\sum_{j_1,\ldots,j_p=1}^\infty \prod_{i=1}^p \| \cP( S_{j_i+r}(z) ) \|_{p \wt p}  \, \E\Big[ \sum_{y\in \cP\cap S_{j_i+r }(z)}	\1{R(y,\cP) \ge \|y-z\| - \sqrt{d}/2 }  \Big]^{1/ (p \wt q)} \\
	&\lesssim \Bigg( \sum_{j=1}^\infty |S_{j+r}(z) |^{1+1/(p \wt q)}  \, \ \p( R(0,\cP\cup \{0\}) \ge j+r - 1 - \sqrt{d}/2 )^{1/ (p \wt q)} \Bigg)^p \\
	&\lesssim \Bigg( \sum_{j=1}^\infty (j+r)^{(d-1) (1 + 1/(p \wt q) )} \ \p( R(0,\cP\cup \{0\}) \ge j+r - 1 - \sqrt{d}/2 )^{1/(p\wt q)} \Bigg)^p \\
	&\lesssim \exp( - a r^{\as}  ),
\end{align*}
for some $a\in\R_+$.

We conclude the second step with the considerations for the first non-local sum in \eqref{E:Approximation5}. Given the condition that additionally $c^* \ge [\wt q (p+1/d) +1]/c_{stab}$ this sum is at most
\begin{align}
	&\E\Big[ \Big|  \sum_{ \substack{ y\in \cPc_n \setminus Q_z, \\ \|y-z\| \ge r } }  \1{ \xi(y,\cP_n) \neq \xi(y,\cP''_{n,z} } 	\Big|^p \Big] \label{E:Approximation8a} \\
	&\le \E\Big[ \cP_n( B(z,r))^p \1{\exists y \in \cPc_n : R(y,\cP) \ge r-\sqrt{d}/2 } \Big] \nonumber \\
	&\lesssim n^{p+1/\wt q} \ \p( R(0,\cP\cup\{0\}) \ge r - \sqrt{d}/2)^{1/\wt q} \nonumber \\
	&\lesssim n^{-1/d},  \nonumber%\label{E:Approximation8}
\end{align}
by the condition in \eqref{E:Approximation3a}.

Otherwise, given the extended stabilizing condition from \eqref{E:Stabilization3} is in force, we can carry out similar calculations as for the second non-local sum in \eqref{E:Approximation6} to see that \eqref{E:Approximation8a} is of order $\exp( - a r^{\as} )$ for some $a\in\R_+$.

\textit{Step 3.}
We rely on the decomposition
\begin{align}
	&\E[ | \Delta(z,n) - \Delta(z,\infty) |^p ] \nonumber \\
	&\lesssim \E[ | \Delta(z,n) |^p ] + \E[ | \Delta(z,\infty) |^p ]  \nonumber \\
	&\lesssim  \E\Big[ \Big|\sum_{y\in \cPc_n \cap Q_z} \Psi(y,\cP_n) \Big| ^p \Big] + \E\Big[ \Big|\sum_{y\in \cP\cap Q_z} \Psi(y,\cP) \Big| ^p \Big]\label{E:Approximation1a} \\
	&\quad + \E\Big[ \Big| \sum_{y\in \cPc_n \setminus Q_z} \1{ \xi(y,\cP_n) \neq \xi(y,\cP''_{n,z})} \Big|^p	\Big] \label{E:Approximation1b}\\
	&\quad + \E\Big[ \Big| \sum_{y\in \cP \setminus Q_z} \1{ \xi(y,\cP) \neq \xi(y,\cP''_{z})} \Big|^p	\Big]. \label{E:Approximation1c}
\end{align}
Each term on the left-hand side of \eqref{E:Approximation1a} is of order $\E[ \cP(Q_0)^p ]<\infty$. 

We continue with the term in \eqref{E:Approximation1c}, which is at most
\begin{align*}
		&\sum_{j_1,\ldots,j_p=1}^\infty \E\Big[ \prod_{i=1}^p	\sum_{y\in (\cP \setminus Q_z) \cap S_{j_i}(z) } \1{ \xi(y,\cP) \neq \xi(y,\cP''_{z}}		\Big]  \\
		&\lesssim \sum_{j_1,\ldots,j_p=1}^\infty \prod_{i=1}^p |S_{j_i}(z)|^{1+1/(p\wt q)} \ \p( R(0,\cP\cup \{0\})\ge j_i-1-\sqrt{d}/2 )^{1/(p\wt q)}.
\end{align*}
And the last right-hand side is uniformly bounded above.

Finally, we come to the term in \eqref{E:Approximation1b}. If the extended stabilization condition from \eqref{E:Stabilization3} is satisfied, one sees that with similar considerations as for the term in \eqref{E:Approximation1c} that \eqref{E:Approximation1b} is uniformly bounded above.

However, if the additional requirement $c^* > (p+1+1/d)/c_{stab}$ is satisfied, we rely on the following upper bound for the term in \eqref{E:Approximation1b}, which is then at most
\begin{align}
	&\sum_{j_1,\ldots,j_p=1}^\infty \E\Big[ \prod_{i=1}^p	\sum_{y\in (\cPc_n \setminus Q_z) \cap S_{j_i}(z) } \1{ \xi(y,\cP_n) \neq \xi(y,\cP''_{n,z}} \Big] \nonumber \\
	&\lesssim\sum_{j_1,\ldots,j_p=1}^\infty  \prod_{i=1}^p |S_{j_i}(z) \cap \Wc_n \setminus Q_z |^{1+1/(p\wt q)} \nonumber \\
	&\qquad\qquad \p( \xi(Y_{j_i},\cP_n \cup \{Y_{j_i}\}) \neq \xi(Y_{j_i},\cP''_{n,z} \cup \{Y_{j_i}\}	)^{1/(p \wt q)}, \nonumber \\
	\begin{split}\label{E:Approximation1d}
	&= \Bigg( \sum_{j=1}^\infty |S_{j}(z) \cap \Wc_n \setminus Q_z |^{1+1/(p\wt q)} \\
	&\qquad\qquad \p( \xi(Y_{j},\cP_n \cup \{Y_{j}\}) \neq \xi(Y_{j},\cP''_{n,z} \cup \{Y_{j}\}	)^{1/(p \wt q)} \Bigg)^p, 
	\end{split}
	&
\end{align}
where $Y_j$ is uniformly distributed on $S_j(z) \cap \Wc_n \setminus Q_z$ and independent of $\cP$ and $\cP'$ for each $j\in\N$. Next, use $\sum_{j=1}^\infty |S_{j}(z) \cap \Wc_n \setminus Q_z |^{1+1/(p\wt q)} \le n^{1+1/(p\wt q)}$ and
\begin{align*}
	& \p( \xi(Y_j,\cP_n \cup \{Y_j\}) \neq \xi(Y_j,\cP''_{n,z} \cup \{Y_j\}	) \\
 	&\le \p( R(0,\cP\cup \{0\}) \ge j-1-\sqrt{d}/2 ) + \p( R(0,\cP\cup \{0\})\ge r ).
\end{align*}
as well as \eqref{E:Approximation3a} to see that \eqref{E:Approximation1d} is uniformly bounded above.

\textit{Step 4.}
Finally, we combine the estimates from the second and third step with \eqref{E:Approximation0} in order to obtain
\begin{align*}
	\Big \| \sum_{z\in B_n} X_{n,z}  \Big\|_p &\le c_{1,p} \Big( \sum_{\substack{z\in B_n,\\ Q_z \subseteq \Wc_n} } \| X_{n,z}^2 \|_{p/2} + \sum_{\substack{z\in B_n,\\ Q_z \not\subseteq \Wc_n} } \| X_{n,z}^2 \|_{p/2} \Big)^{1/2} \\
	&\le c_{1,p} \Big( \sum_{\substack{z\in B_n,\\ Q_z \subseteq \Wc_n} } c_3 n^{-1/d} + \sum_{\substack{z\in B_n,\\ Q_z \not\subseteq \Wc_n} } c_4  \Big)^{1/2} \\
	&\lesssim r^{1/2} n^{(d-1)/2d} .
\end{align*}
In particular, relying on this last result
\begin{align*}
	\Big\| \max_{k\le n} \Big| \sum_{z\in B_k} X_{z,k} \Big| \Big\|_p \lesssim r^{1/2} n^{(d-1)/2d + 1/p}.
\end{align*}
Applying Lemma~\ref{L:LILTool}, we obtain
\[
	 \sum_{z\in B_n} X_{z,n} = O_{a.s.}(r^{1/2} n^{(d-1)/2d + 1/p} (\log n)^{1+\delta} )
\]
for all $\delta>0$. This completes the proof.
\end{proof}

\begin{proposition}\label{P:AbstractLIL}[An abstract LIL and CLT]
Consider the martingale differences induced by \eqref{E:MDS2}. Then
\begin{align}
		& \limsup_{n\to\infty} \frac{\pm \sum_{z\in B_n} \E[ \Delta(z,\infty) | \cF_z ] }{ \sqrt{2 n \log \log n} } = \sigma \quad a.s., \label{E:LIL1} \\
		& n^{-1/2}  \sum_{z\in B_n} \E[\Delta(z,\infty) | \cF_z] \Rightarrow \cN(0,\sigma^2), \quad n\to\infty, \label{E:CLT1}
\end{align}
where $\sigma^2 = \E[ \E[ \Delta(0,\infty) | \cF_0]^2 ] < \infty $. 
$\sigma^2>0$ if $\E[ \Delta(0,\infty) | \cF_0]$ has a non-degenerate distribution.

In particular, in conjunction with Proposition~\ref{P:Approximation}
\begin{align}
		& \limsup_{n\to\infty} \frac{\pm (\sum_{y\in \cPc_n} \Psi(y,\cP_n) - \E[\sum_{y\in \cPc_n} \Psi(y,\cP_n)] )}{ \sqrt{2 n \log \log n} } = \sigma \quad a.s., \label{E:LIL2}  \\
		& n^{-1/2}  \Big( \sum_{y\in \cPc_n} \Psi(y,\cP_n) - \E\Big[\sum_{y\in \cPc_n} \Psi(y,\cP_n)\Big] \Big) \Rightarrow \cN(0,\sigma^2), \quad n\to\infty. \label{E:CLT2}
\end{align}
\end{proposition}
\begin{proof}[Proof of Proposition~\ref{P:AbstractLIL}]
We only prove \eqref{E:LIL1} and \eqref{E:CLT1}. The amendment given in \eqref{E:LIL2} and \eqref{E:CLT2} is then a direct consequence of \eqref{E:ApproximationMain1} from Proposition~\ref{P:Approximation}.

We rely on the LIL given in Krebs \cite[Theorem~3.1]{krebs2021LIL} for \eqref{E:LIL1} and the normal approximation given in Krebs \cite[Proposition 3.1]{krebs2021LIL} for \eqref{E:CLT1}. The proof is then immediate if we use additionally the uniform bounded moments condition from \eqref{E:ApproximationMain0} for the random variables $\Delta(z,\infty)$ and the exponential stabilizing condition from \eqref{E:DeltaInfty1}.
\end{proof}

\begin{proof}[Proof of Theorem~\ref{T:LILQuant}]
The result follows from Proposition~\ref{P:AbstractLIL} and the Bahadur representation from Theorem~\ref{T:BahadurPoissonUnif}.
\end{proof}

\subsection{The asymptotic normality}
\begin{lemma}\label{L:FHatAlt}
%Let $r=(c^* \log n)^{1/\as}$ for $c^* \ge 4/c_{stab}$. Then
We have
\begin{align}
	\wh F_n(x) - F(x) &=   \frac{1}{\Mc_n} \Big\{ \sum_{y\in \cPc_n} \1{ \xi(y,\cP_n) \le x } - \E\Big[ \sum_{y\in \cPc_n} \1{ \xi(y,\cP_n) \le x } \Big] \Big\} \nonumber\\	
	&\quad - F(x) \frac{\Mc_n - \E[ \Mc_n ]}{\Mc_n}+ O\Big( \frac{ |\Wc_n| }{ \Mc_n } n^{-4} \Big). \label{E:FHatAlt1}
	\end{align}
\end{lemma}
\begin{proof}
The proof is a direct consequence of the Slyvniak-Mecke formula and the fact that the scores stabilize exponentially. We subtract and add the term $\E[  \sum_{y\in \cPc_n} \1{ \xi(y,\cP_n) \le x } ] / \Mc_n$ to the left-hand side of \eqref{E:FHatAlt1}, then the first term on the right-hand side of \eqref{E:FHatAlt1} is immediate and the second term follows from
\begin{align*}
	&\frac{1}{\Mc_n} \E\Big[  \sum_{y\in \cPc_n} \1{ \xi(y,\cP_n) \le x } \Big]  - \p( \xi(0,\cP\cup \{0\})\le x) \\
	&= F(x) \frac{ |\Wc_n| - \Mc_n}{\Mc_n} + O\Big( \frac{ |\Wc_n| }{ \Mc_n } e^{-c_{stab} r^{\as} } \Big)
\end{align*}
in conjunction with the definition of $r$.	
\end{proof}

\begin{proof}[Proof of Theorem~\ref{T:WeakConvergence}]
Provided $( \sqrt{n}(\wh F_n(x) - F(x)): x\in I)_n \Rightarrow (W(x):x\in I)$ in $D(I)$, it follows from Theorem~\ref{T:BahadurPoissonUnif} (i) that $(\sqrt{n}(\wh \psi_{p,n} - \psi_p): p_0 \le p \le p_1)_n \Rightarrow (W(x)/f(x): \psi_{p_0}\le x\le \psi_{p_1})$ in the Skorokhod topology for $I=[\psi_{p_0},\psi_{p_1}]$.

In order to prove now the first statement, we proceed in three steps. In the first step, we establish the convergence of the finite dimensional distributions. In the second step, we verify the tightness. Finally, in the third step, we show the continuity property of the sample paths of $W$.

\textit{Step 1.} Let $m\in\N_+$, $x_1,\ldots,x_m\in\R$ and $a_1,\ldots,a_m \in\R$. We use the score $\Psi(y,\cP_n)$, which is defined in \eqref{E:FidiScores}, for $y\in\cPc_n$. Then we have by Lemma~\ref{L:FHatAlt}
\begin{align*}
		&\sqrt{n} \Big( \sum_{i=1}^m a_i (\wh F_n(x_i) - F(x_i))	\Big) \\
		&= \frac{n}{\Mc_n} \cdot \frac{1}{\sqrt{n}}  \Big( \sum_{y\in \cPc_n} \Psi(y,\cP_n) - \E\Big[ \sum_{y\in\cPc_n} \Psi(y,\cP_n)\Big] \Big) + o_{a.s.}(1).
\end{align*}
In order to show that this sequence converges to a normal distribution, it is sufficient to consider the convergence of
\begin{align}
		&  \frac{1}{\sqrt{n}} \Big( \sum_{y\in \cPc_n} \Psi(y,\cP_n) - \E\Big[ \sum_{y\in\cPc_n} \Psi(y,\cP_n) \Big] \Big). \label{E:Normality}
\end{align}
We rely on Proposition~\ref{P:AbstractLIL} to obtain that \eqref{E:Normality} converges weakly to $\cN(0,\sigma^2)$ for some $\sigma^2<\infty$. This shows the first step.

\textit{Step 2.}
Define for $z\in B_n$
\begin{align}
	&\wt\Delta_r(z,n) \nonumber\\
	\begin{split}\label{E:Tight0}
	&\coloneqq \sum_{y\in \cPc_n  } \1{ \xi_r(y,\cP_n) \in (u,x]} - \sum_{y'\in \cPc_{n,z}  } \1{ \xi_r(y',\cP''_{n,z}) \in (u,x] }
 \\	
 &= \sum_{y\in \cPc_n  \cap O_z  } \1{ \xi_r(y,\cP_n) \in (u,x]} - \sum_{y'\in \cPc_{n,z} \cap O_z  } \1{ \xi_r(y',\cP''_{n,z}) \in (u,x] },
	\end{split}
\end{align}
where $O_z=B(z,r+\sqrt{d}/2)$. The last equality holds because the scores $\xi_r$ stabilize within a distance of $r+\sqrt{d}/2$ to $z$.

In order to verify the tightness, we rely on the criterion of Bickel and Wichura \cite{bickel1971convergence}, which is a refined and generalized argument of Billingsley \cite{billingsley1968convergence}. We proceed as follows: First, we reduce the setting from the interval $I=[a,b]$ to a sequence of grids $(\Gamma_n)_n$, where $\Gamma_n = \{a+ k (b-a)/n: k\in \{0,\ldots,n\} \}$. Second, we verify the moment condition for the processes $(\sqrt{n}(\wh F_n(x)-F(x)): x\in \Gamma_n)_n$.

Let $x,u\in I$. Let $\ul x, \ol x\in \Gamma_n$ be the points closest to $x$ such that $\ul x \le x \le \ol x$. Let $\ul u \le u \le \ol u$ be defined similarly for $u$. Then using the monotonicity of $F$ and that the derivative of $F$ exists on $I$ and is uniformly bounded, we achieve on the one hand
\begin{align*}
		&\sqrt{n}\Big( \big( \wh F_n(x) - F(x)) - (\wh F_n(u)-F(u) \big) \Big)\\
		 &\le \sqrt{n}\Big( \big( \wh F_n(\ol x) - F(\ol x)) - (\wh F_n(\ul u)-F(\ul u) \big) \Big) + \frac{ \sup_{x\in I} f(x) (b-a)}{\sqrt{n}}.
\end{align*}
On the other hand,
\begin{align*}
		&\sqrt{n}\Big( \big( \wh F_n(x) - F(x)) - (\wh F_n(u)-F(u) \big) \Big)\\
		 &\ge \sqrt{n}\Big( \big( \wh F_n(\ul x) - F(\ul x)) - (\wh F_n(\ol u)-F(\ol u) \big) \Big) - \frac{ \sup_{x\in I} f(x) (b-a)}{\sqrt{n}}.
\end{align*}
Thus,
\begin{align*}
	\sup_{x,u\in I} &\sqrt{n}\Big| \big( \wh F_n(x) - F(x)) - (\wh F_n(u)-F(u) \big)  \Big| \\
	&\le \max_{x,u\in \Gamma_n} \sqrt{n}\Big| \big( \wh F_n(x) - F(x)) - (\wh F_n(u)-F(u) \big)  \Big| + O(n^{-1/2}).
\end{align*}
Hence, we can in the following verify the moment condition for intervals $[u,x]$ from $\Gamma_n$. For this purpose, we rely on the following decomposition, which is for each point $x$ as follows.
\begin{align}
	  \wh F_n(x) - F(x)  &= (\wh F_n(x) - \wh F_{n,r} (x) ) \label{E:Tight1} \\
	  \begin{split}\label{E:Tight2}
	  &\quad + (\Mc_n)^{-1}\Big(   \sum_{y\in\cPc_n} \1{\xi_r(y,\cP_n) \le x} \Big) \\
	  &\qquad\qquad -  \E\Big[ \sum_{y\in\cPc_n} \1{\xi_r(y,\cP_n) \le x} \Big] \Big) \1{ \Mc_n >0}
	  \end{split} \\
	  &\quad + \Big( \frac{|\Wc_n|}{\Mc_n}-1 \Big) \p( \xi_r(0,\cP \cup \{0\}) \le x ) \1{\Mc_n >0} \label{E:Tight3} \\
	  &\quad + \p( \xi_r(0,\cP \cup \{0\}) \le x ) (\1{\Mc_n >0} -1) \label{E:Tight3b} \\
	  &\quad +  \big( \p( \xi_r(0,\cP \cup \{0\}) \le x ) - \p( \xi(0,\cP \cup \{0\}) \le x ) \big). \label{E:Tight4}
\end{align} 
In the following, let $x,u \in \Gamma_n$, $u\le x$. We show there is a $C\in\R_+$, which does not dependent on $n$, with the property
\begin{align}\label{E:Tight5}
	n^{2} \ \E[ |\wh F_n(x) - F(x) - \wh F_n(u) + F(u)|^4 ] \le C |x-u|^{3/2}.
\end{align}
To this end, we use the decomposition from \eqref{E:Tight1} to \eqref{E:Tight4} and show that each difference induced by a term from \eqref{E:Tight1} to \eqref{E:Tight4} is of order $|x-u|^{3/2}$ uniformly in $n\in\N$.

We begin with the difference induced by the term in \eqref{E:Tight1}. Clearly, $|\wh F_n(x) - \wh F_{n,r}(x)|^4\le |\wh F_n(x) - \wh F_{n,r}(x)|$ for each $x\in\R^d$. Moreover, relying on the estimate in \eqref{E:Lemma1.0} from Lemma~\ref{Lemma1}, we have
\begin{align*}
	\E[ \sup_{x\in\R} |\wh F_n(x) - \wh F_{n,r}(x)| ] &\le \E\Bigg[ \frac{1}{\Mc_n} \sum_{y\in \cPc_n} h(y,\cPc_n  \cup \cPd_n ) \\
	&\qquad\qquad \Bigg( \1{2 > \frac{\E[\Mc_n]}{\Mc_n} } + \1{2 \le \frac{\E[\Mc_n]}{\Mc_n} }\Bigg) \Bigg] \\
	&\le \frac{2}{\E[\Mc_n]} \E\Big[ \sum_{y\in \cPc_n} h(y,\cPc_n  \cup \cPd_n ) \Big] \\
	& + \E[ \1{ \Mc_n \le \E[ \Mc_n ] /2 } ].
\end{align*}
The first term can be treated with the Slivnyak-Mecke formula as in \eqref{E:Lemma1.1}, the second term decays exponentially in $n$ by Lemma~\ref{Lemma0}. Consequently, there is a $C\in\R_+$ such that
\begin{align*}
		n^{2} \ \E[ | \wh F_n(x) - \wh F_{n,r}(x) - \wh F_n(u) + \wh F_{n,r}(u) |^4 ] \le C n^{-2} \lesssim |x-u|^2 ,
\end{align*}
where the last conclusion follows because $n^{-1}\lesssim |x-u|$ for $x,u\in \Gamma_n$.

We continue with the difference induced by \eqref{E:Tight3}. 
\begin{align}
	&n^{2} \ \E\Big[ \Big( \frac{ |\Wc_n|}{\Mc_n} -1 \Big)^4 \1{\Mc_n>0}   \Big] \ \p( \xi_r(0,\cP \cup \{0\} ) \in (u, x] )^4 \nonumber \\
	&\lesssim \frac{ n^{2} }{ |\Wc_n|^4 } \ \E\Big[ \Big| |\Wc_n| - \Mc_n \Big|^4 \Big(\frac{ |\Wc_n| }{\Mc_n}\Big)^4 \1{\Mc_n>0} \Big] \Big( |x-u|^4 + n^{-16} \Big), \label{E:Tight5b}
\end{align}
using that uniformly in $x$
\[
	| \p( \xi_r(0,\cP\cup \{0\}) \le x ) -  \p( \xi(0,\cP \cup \{0\})  \le x ) | = O(n^{-4}).
	 \]
Now, $\E[ | \Mc_n - |\Wc_n| |^4] = |\Wc_n| (1+3|\Wc_n| )$ and $|\Wc_n|^4 \le n^4$. Furthermore, the quotient $|\Wc_n| / \Mc_n$ is less than 2 with a probability of at least $1-2e^{-|\Wc_n|/12}$. So, it is a routine to show that \eqref{E:Tight5b} is of order $|x-u|^4$ (uniformly in $n$).

The difference induced by the term in \eqref{E:Tight3b} is 
\[
	n^2 \ \E[ | \p( \xi_r(0,\cP \cup \{0\} ) \in (u,x] ) |^4 \1{\Mc_n = 0} ] \le n^2 \exp( - |\Wc_n|)
\]
and $o(|x-u|^{3/2})$. Further, we see for the difference induced by the term in \eqref{E:Tight4}
\begin{align*}
	&n^{2} \ \E[ | \p( \xi_r(0,\cP \cup \{0\} ) \in (u, x] ) - \p( \xi(0,\cP\cup \{0\} ) \in (u,x] )  |^4 ] \lesssim n^{-12},
\end{align*}
which is  $o(|x-u|^{3/2})$.

Finally, we come to the main difference, which is induced by \eqref{E:Tight2}.
\begin{align}\begin{split}\label{E:Tight6}
	 &\E\Big[ \Big| \frac{\sqrt{n}}{\Mc_n}	\Big|^4 \ \Big| \sum_{y\in\cPc_n} \1{\xi_r(y,\cP_n) \in (u, x]} \\
	 &\qquad\qquad\qquad - \E[ \sum_{y\in\cPc_n} \1{\xi_r(y,\cP_n) \in (u, x]} ] \Big|^4  \1{\Mc_n >0 }	\Big].
\end{split}\end{align}
We consider the event
$A_n = \{ |\Wc_n| - |\Wc_n|^{3/4} \le \Mc_n\le |\Wc_n| + |\Wc_n|^{3/4}\} $.
Then $\p(A_n^c)$ decays exponentially in $n$. Consequently, for the moment condition, it is enough to restrict the random variable in the expectation in \eqref{E:Tight6} to the event $A_n$. In particular, on $A_n$, $\Mc_n$ is at least 1 once $n$ is larger than some critical $n_0$, which we can assume in the following.
In this case and up to a multiplicative constant the expectation is bounded above by
\begin{align}
	&n^{-2} \ \E\Big[ \Big| \sum_{y\in\cPc_n} \1{\xi_r(y,\cP_n) \in (u, x]} - \E[ \sum_{y\in\cPc_n} \1{\xi_r(y,\cP_n) \in (u, x]} ] \Big|^4 	\Big] \nonumber \\
	&=n^{-2} \ \E\big[ | \sum_{z\in B_n} \E[ \wt\Delta_r(z,n) | \cF_z ] |^4 \big], \label{E:Tight6b}
\end{align}
where $\wt\Delta_r(z,n)$ is given in \eqref{E:Tight0}.
By Burkholder's inequality \eqref{E:Tight6b} is at most
\begin{align}
	&n^{-2} \ \E\big[ | \sum_{z\in B_n} \E[ \wt\Delta_r(z,n) | \cF_z ]^2 |^2 \big] \nonumber \\
	&= n^{-2} \sum_{\substack{z, z' \in B_n:\\ \|z-z'\|_\infty\le 2r+3\sqrt{d}}} \E\big[ \E[ \wt\Delta_r(z,n) | \cF_z ]^2 \ \E[\wt\Delta_r(z',n) | \cF_{z'} ]^2  \big] \label{E:Tight7}
 \\
	&\quad + n^{-2} \sum_{\substack{z, z' \in B_n:\\ \|z-z'\|_\infty > 2r+3\sqrt{d}}} \E[ \E[\wt\Delta_r(z,n) | \cF_z ]^2 ] \ \E[\E[\wt\Delta_r(z',n) | \cF_{z'} ]^2  ].  \label{E:Tight8}
\end{align}
The sum in \eqref{E:Tight7} collects the products of squared martingale differences whose locations $z,z'$ are close, whereas the sum in \eqref{E:Tight8} collects the independent products because the corresponding locations are far enough apart.

We begin with the sum in \eqref{E:Tight7} which we can easily treat with the H{\"o}lder inequality by considering the fourth moments. Choose $\ol c>0$ such that we have for the event $A_{n,z} = \{ \cPc_n(O_z) \le  |O_z|  + \ol c |O_z|^{3/4} \}$ that $\p(A_{n,z}^c) \lesssim n^{-4}$ uniformly in $z\in B_n$. We have 

\begin{align*}
	&\E\big[ \E[ \wt\Delta_r(z,n) | \cF_z ]^4 \big] \\
	&\le 16 \ \E\Big[ \Big| \sum_{y\in\cPc_n \cap O_z} \1{ \xi_r(y,\cP_n)\in (u,x] } \Big|^4 \ \big\{ \1{A_{n,z}} + \1{ A_{n,z}^c}\Big\} 	\Big] \\
	&\le 16 \big(|O_z| + |O_z|^{3/4} \big)^3 \ \E\Big[  \sum_{y\in\cPc_n \cap O_z} \1{ \xi_r(y,\cP_n)\in (u,x] }  \Big] \\
	&\quad +  16 \ \E\big[ \cPc_n(O_z)^8 \big]^{1/2} \ \p( A_{n,z}^c)^{1/2}.
\end{align*}
The last term is of order $r^{4d} n^{-2}$ and is negligible. The first term is of order $r^{4d} |x-u| + r^{4d} n^{-4} \lesssim r^{4d} |x-u|$. Both upper bounds do not depend on the position $z$. So, the sum in \eqref{E:Tight7} is at most (up to a multiplicative constant)
\[
	n^{-2} \sum_{\substack{z, z'\in B_n:\\ \|z-z'\|_\infty\le 2r+3\sqrt{d}}} r^{4d} |x-u| \lesssim n^{-1} r^{5d} |x-u| \lesssim  |x-u|^{3/2}
\]
because $r$ grows logarithmically in $n$. So, \eqref{E:Tight7} satisfies the moment condition.

We conclude with the sum in \eqref{E:Tight8}. In order to establish the moment condition for this sum, we rely on Lemma~\ref{Lemma4}. For each $p>1$ there is a constant $C_p\in\R_+$ such that for all $z\in B_n$ and $n\in\N$
\begin{align}
 	&\E[ \E[\wt\Delta_r(z,n) | \cF_z ]^2 ] \le \E[\wt\Delta_r(z,n)^2 ] \nonumber \\
	&=\E\Big[ \Big(\sum_{y\in \cPc_n  } \1{ \xi_r(y,\cP_n) \in (u,x]} - \sum_{y'\in \cPc_{n,z}  } \1{ \xi_r(y',\cP''_{n,z}) \in (u,x] } \Big)^2 \Big] \nonumber \\
	&\le C_p \ \p(  \xi_r(0,\cP \cup \{0\} ) \in (u,x] )^{1/p} \lesssim |x-u|^{1/p} + n^{-4/p} . \nonumber
\end{align}
Choosing $p = 4/3$, there is a constant $C\in\R_+$ such that \eqref{E:Tight8} is at most
\[
	C n^{-2} \sum_{\substack{z, z' \in B_n:\\ \|z-z'\|_\infty > 2r+3\sqrt{d}}}  (|x-u|^{3/4} + n^{-3})^2 \lesssim |x-u|^{3/2}. 
\]
This yields that the difference which is induced by the term in \eqref{E:Tight2} is at most $ C |x-u|^{3/2}$ for a $C\in\R_+$.
Combining all estimates, shows that \eqref{E:Tight5} holds.

\textit{Step 3.} Recall that $F$ admits a bounded density $f$ on $I$ by assumption. Since $W$ is a Gaussian process, we have for each $k\in\N$
\[
	\E[ |W(t)-W(s)|^{2k} ] = \prod_{i=1}^k (2i-1) \E[ |W(t)-W(s)|^2 ]^k.
\]
To obtain an explicit expression of the covariance, we  rely on \eqref{E:MDS2} and define
\begin{align*}
	\Delta_t(0,\infty) &= \sum_{y\in \cP\setminus Q_0} \Big(\1{ \xi(y,\cP)\le t} - \1{\xi(y,\cP''_0)\le t} \Big)\\
	&\quad +	\sum_{y\in \cP\cap Q_0} \Big( \1{ \xi(y,\cP)\le t} - F(t) \Big) \\
	&\quad - \sum_{y\in \cP''_0 \cap Q_0} \Big( \1{ \xi(y,\cP''_0)\le t} - F(t)\Big).
\end{align*}
Then by Proposition~\ref{P:AbstractLIL}
\begin{align*}
	\E[ |W(t)-W(s)|^2 ] &= \E[ \E[\Delta_t(0,\infty) - \Delta_s(0,\infty) | \cF_0]^2 ] \\
	&\le \E[ |\Delta_t(0,\infty)-\Delta_s(0,\infty)|^2 ].
\end{align*}
Finally, in conjunction with the existence of $f$ on $I$, we can apply the amendment of Lemma~\ref{Lemma4} in \eqref{E:Lemma4.0B}: For each $p>1$ there is a $C_p\in\R_+$ such that
\[
	\E[ |\Delta_t(0,\infty)-\Delta_s(0,\infty)|^2 ] \le C_p |t-s|^{1/p}.
\]
Consequently, by the Kolmogorov-Chentsov continuity theorem and given $k$ and $p$, there is a continuous modification of $W$ on $I$ which is H{\"o}lder continuous with exponent $(k/p-1)/k$. Obviously, we find for each $\beta\in (0,1/2)$, parameter $k\in\N$ and $p>1$ such that $(k/p-1)/k > \beta$. This completes the proof.
\end{proof}

%\newpage
%\appendix
\section{Some useful results}\label{AppendixLemmas}

\begin{lemma}\label{Lemma0}
If $\lambda,t>0$ and $Z\sim\poi(\lambda)$, then 
$$
	\p( |Z-\lambda| \ge t) \le 2\exp[ - t^2/(2(\lambda+t)) ].
$$
In particular, $\p( Z/\lambda \ge 2 ) \le 2 e^{-\lambda/4}$ and $\p( Z/\lambda \le 1/2 ) \le 2 e^{-\lambda/12}$.
\end{lemma}

\begin{lemma}\label{Lemma1}
Let $\tau, k\ge 0$ and $r= (c^* \log n)^{1/\as}$ for $c^* \ge (\tau+k)/c_{stab}$. 
\begin{itemize} \setlength{\itemsep}{2mm}
	\item [(i)] For each $\epsilon\in \R_+$ we have  
	$$	
		\p \big( \sup_{x\in\R} |\wh F_n(x) - \wh F_{n,r}(x) | \ge \epsilon n^{-k} \big) = O(n^{-\tau}).
	$$
	\item [(ii)] For $x\in \R$, let
	\begin{align*}
		\wt F_{n}(x) &=  \frac{1}{M_n} \sum_{y\in \cP_n} \1{ \xi (y,\cP_n) \le x}, \\
		\wt F_{n,r}(x) &=  \frac{1}{M_n} \sum_{y\in \cP_n} \1{ \xi_r (y,\cP_n) \le x}.
	\end{align*}
	There is a $c_{\tau} \in \R_+$, depending on the choice of $c^*$, such that 
	$$\p\big(\sup_{x\in\R} |\wt F_n(x) - \wt F_{n,r}(x) | \ge c_{\tau}  (n^{-k} \vee r n^{-1/d}) \big) = O(n^{-\tau}).$$
\end{itemize}
\end{lemma}
\begin{proof}
We prove the statement in (i) in detail. We have
\begin{align}\label{E:Lemma1.0}
	&\sup_{x\in\R} |\wh F_n(x) - \wh F_{n,r}(x)| \le \frac{1}{\Mc_n} \sum_{y\in \cPc_n} h(y,\cPc_n   \cup \cPd_n ),
\end{align}
where we set
\begin{align*}
	h(y,\cPc_n \cup \cPd_n ) &= \1{ \xi_r (y,\cPc_n \cup \cPd_n ) \neq \xi (y,\cPc_n \cup \cPd_n ) }.
\end{align*}
Then by the Slivnyak-Mecke formula (see e.g.\ \cite[Theorem 4.4]{last2017lectures}),
\begin{align}
	 \E\Big[ \sum_{y\in \cPc_n} h(y,\cPc_n \cup \cPd_n )  \Big] &= \int_{\Wc_n}  \E[ h(y,\cPc_n \cup \cPd_n \cup \{y\} )  ] \diff y \nonumber \\
	 &= \int_{\Wc_n}  \p( \xi_r(y,\cP_n \cup \{y\} ) \neq \xi(y,\cP_n \cup \{y\} ) ) \diff y \nonumber \\
	 &= \int_{\Wc_n}  \p( \xi_r(y,\cP \cup \{y\} ) \neq \xi(y,\cP \cup \{y\} ) ) \diff y \nonumber \\
	 &\le |\Wc_n| \ \p( R(0,\cP \cup \{ 0\} ) > r ) \label{E:Lemma1.1}
\end{align}
because $\cP$ and $\cP_n$ agree on $B(y,r)$ by construction. Using the exponential stabilization of the score functions from \eqref{E:Stabilization2}, \eqref{E:Lemma1.1} is bounded above by 
\begin{align*}	
	&|\Wc_n| \ C_{stab} \exp( - c_{stab} r^{\as}) = C_{stab} \ \E[\Mc_n] \ n^{ - c_{stab} c^* } .
\end{align*}
Let $\epsilon>0$. We distinguish the cases whether the quotient $|\Wc_n|/ \Mc_n$ is greater or at most 2. Then using Markov's inequality and Lemma~\ref{Lemma0}, we have
\begin{align*}
	&\p\big( \sup_{x\in\R} |\wh F_n(x) - \wh F_{n,r}(x) | \ge \epsilon n^{-k} \big) \\
	 &\le  \p\big( |\Wc_n|^{-1} \sum_{y\in \cPc_n} h(y,\cPc_n\cup \cPd_n ) \ge  \epsilon n^{-k} / 2 \big) + \p\big( |\Wc_n| /  \Mc_n > 2 \big) \\
	 &\le C_{stab} \frac{2 n^{k}}{\epsilon}  \ n^{ - c_{stab} c^* } + 2 \exp \Big( - \frac{ |\Wc_n| }{ 12} \Big) = O(n^{-\tau}).
\end{align*}
This shows the first statement.

The proof of (ii) is very similar to (i) and we only give a sketch. We use additionally the well-known concentration of a Poisson random variable around its mean: By Lemma~\ref{Lemma0} $\p( | M_n - \E[M_n] | \ge \E[M_n]^{1/2 + \alpha} )$ decays exponentially for each $\alpha>0$;
the same is true for the concentration of $(M_n - \Mc_n)_n$ around its mean. 

Consequently, using that $\E[M_n] = n$ and $\E[M_n - \Mc_n ] = O(n^{(d-1)/d} r)$, we arrive at (for the choice $\alpha=1/4$)
\begin{align}\label{E:Lemma1.2}
	\frac{M_n - \Mc_n}{M_n} = O_{a.s.}\Big( 	\frac{n^{(d-1)/d}r + (n^{(d-1)/d}r)^{1/2 + \alpha} }{n - n^{1/2+\alpha}} \Big) = O_{a.s.}\Big( 	\frac{r }{n^{1/d} } \Big).
\end{align}
Next, use the estimate
\[
	| \wt F_n(x) - \wt F_{n,r}(x) | \le \frac{\Mc_n}{M_n} | \wh F_n(x) - \wh F_{n,r}(x) | + \frac{M_n - \Mc_n}{M_n}
\]
in combination with \eqref{E:Lemma1.2} and the result from the first part. This yields the second statement. Note that in this case the choice of $c_{\tau}$ depends on $c^*$.
\end{proof}

\begin{lemma}\label{Lemma2}
Let $Y$ be uniformly distributed on $\Wc_n$ and let $\cV_{k,n}$ be a $k$-binomial process with uniform density on $W_n$, which is additionally independent of $Y$, for each $k\in [n-n^{3/4}, n+n^{3/4} ]$. Then there is a constant $C\in\R$, which depends on $d$ but not on $n$, such that for each $I\subseteq \R$
\begin{align}\begin{split}\label{E:Lemma2.1}
	\sup_{k\in [n-n^{3/4}, n+n^{3/4} ]} &\big| \p( \xi_r( 0, \cP\cup \{0\} )  \in I ) \\
	&\quad- \p(\xi_r( Y, \cV_{k,n} \cup \{Y\} ) \in I ) \big| \le C r^{2d}n^{-1}.
\end{split}\end{align}
\end{lemma}
\begin{proof}
Since we consider the functional only in an $r$-neighborhood of $Y$ and since $B(Y,r)$ lies entirely inside $W_n$ (because $Y$ lies in $\Wc_n$), we rely on a coupling as follows: We condition on a realization of $Y$. Let $U_1,U_2\ldots$ be an infinite sequence of iid random variables which are uniformly distributed on $B(Y,r)$ conditional on $Y$. Let $N'_k  \sim \text{Bin}(k, |B(0,r)|/n)$ and let $N\sim\poi( |B(0,r)| )$. Then the probabilities on the left-hand side of \eqref{E:Lemma2.1} equal
\[
	| \p( \xi_r(Y, \{U_1,\ldots,U_N\}\cup\{Y\} ) \in I ) - \p( \xi_r(Y, \{U_1,\ldots,U_{N'_k} \}\cup\{Y\} ) \in I) |.
\]
Using Le Cam's theorem on the binomial approximation, there is a coupling of $N$ and $N'_k$ such that 
$$
	\p( N \neq N'_k ) \le 2 \sum_{j=1}^k \Big( \frac{|B(0,r)|}{n} \Big)^2 \lesssim \frac{kr^{2d}}{n^2}, 
$$
see also \cite[Theorem 5.1]{den2012probability}. This completes the proof.
\end{proof}

\begin{lemma}\label{Lemma3}
Let $\fA\subseteq \R^d$ be bounded and measurable. Let $I\subseteq\R$ be an interval. Put $\lambda = |\Wc_n \cap \fA|$ and let $Z\sim \poi(\lambda)$. Then for all $q>0$
\begin{align}
\begin{split}\label{E:Lemma3.0}
	\E\Big[ \Big(\sum_{y\in \cPc_n \cap \fA } \1{ \xi_r(y,\cP_n) \in I} \Big)^2 \Big] &\le \lambda (1+q) \ \p(  \xi_r(0,\cP \cup \{0\} ) \in I ) \\
	&\quad + \lambda^2 \p( Z \ge \floor{q}).
\end{split}\end{align}
In particular, let $\fA = B(z,r)$, $c^*>0$ and $V\sim\poi (|B(0,r)|)$. Then
\begin{align*}
	\E\Big[ \Big(\sum_{y\in \cPc_n \cap B(z,r) } \1{ \xi_r(y,\cP_n) \in I} \Big)^2 \Big] &\lesssim c^* r^{2d} \ \p(  \xi_r(0,\cP \cup \{0\} ) \in I ) \\
	&\quad + r^{2d} \ \p( V > c^* r^d - 1 ).
\end{align*}
\end{lemma}
\begin{proof}
The left-hand side is 
\begin{align}
\begin{split}\label{E:Lemma3.1}
	&\E\Big[ \sum_{y\in \cPc_n \cap \fA  } \1{ \xi_r(y,\cP_n) \in I}  \Big] \\
	&\quad + \E\Big[ \sum_{ \substack{y,y'\in \cPc_n \cap \fA ,\\ y\neq y'} }  \1{ \xi_r(y,\cP_n) \in I} \ \1{ \xi_r(y',\cP_n) \in I}  \Big].
\end{split}
\end{align}
By the Slivnyak-Mecke formula, the first term in \eqref{E:Lemma3.1} equals (see also the calculations preceeding \eqref{E:Lemma1.1})
$$
	\lambda \ \p( \xi_r(0,\cP\cup\{0\}) \in I ) .
$$
For the second term, let $Y,Y',X_1,\ldots,X_{k-2}$ be iid with a uniform distribution on $\Wc_n \cap \fA$. Set $\mX_{k-2} = \{ X_1,\ldots,X_{k-2} \}$. The second term equals
\begin{align*}
		&\sum_{k\ge 2} e^{-\lambda} \frac{\lambda^k}{k!} k(k-1) \E\Big[ \1{ \xi_r(Y,\mX_{k-2} \cup \{Y,Y'\}\cup \cPd_n) \in I} \\
		&\qquad\qquad\qquad\qquad\qquad \1{ \xi_r(Y',\mX_{k-2} \cup \{Y,Y'\}\cup \cPd_n) \in I}	\Big] \\
		&\le \lambda \ \E\Big[ \cPc_n( B(z,r)) \1{ \xi_r(Y,\cPc_n \cup \{Y\}\cup \cPd_n) \in I} \Big] \\
		&\le \lambda q  \p( \xi_r(0,\cP\cup\{0\}) \in I ) + \lambda \E\Big[ \cPc_n( B(z,r)) \1{ \cPc_n( B(z,r)) > q } \Big] \\
		&\le \lambda q  \p( \xi_r(0,\cP\cup\{0\}) \in I ) +  \lambda \sum_{k=\floor*{q}+1}^\infty  e^{-\lambda} \frac{\lambda^{k}}{k!} k \\
		&\le \lambda q  \p( \xi_r(0,\cP\cup\{0\}) \in I ) +  \lambda^2 \sum_{k=\floor*{q}+1}^\infty  e^{-\lambda} \frac{\lambda^{k-1}}{(k-1)!} \\  
		&\le \lambda q \p( \xi_r(0,\cP\cup\{0\}) \in I ) + \lambda^2 \ \p( Z \ge \floor{q} ).
\end{align*}
This completes the proof.
\end{proof}

\begin{lemma}\label{Lemma4}
Let $I\subseteq\R$ be an interval. Let $p>1$. There is a $C_p\in\R_+$ such that uniformly in $z\in B_n$ and $n\in\N$
\begin{align}
\begin{split}\label{E:Lemma4.0}
	&\E\Big[ \Big(\sum_{y\in \cPc_n  } \1{ \xi_r(y,\cP_n) \in I} - \sum_{y'\in \cPc_{n,z}  } \1{ \xi_r(y',\cP''_{n,z}) \in I} \Big)^2 \Big] \\
	&\le C_p \ \p(  \xi_r(0,\cP \cup \{0\} ) \in I )^{1/p}.
\end{split}\end{align}
A similar statement is valid for the score functional $\xi$: For each $p>1$ there is a $C_p\in\R_+$ such that
\begin{align}\begin{split}\label{E:Lemma4.0B}
	&\E\Big[ \Big(\sum_{y\in \cP\setminus Q_0} \1{ \xi(y,\cP) \in I } - \1{\xi(y,\cP''_0)\in I} \Big)^2 \Big]\\
	&\quad +	\E\Big[ \Big( \sum_{y\in \cP\cap Q_0}  \1{ \xi(y,\cP) \in I }  \Big)^2 \Big] \le C_p \p(  \xi(0,\cP \cup \{0\} ) \in I )^{1/p}.
\end{split}\end{align}
\end{lemma}
\begin{proof}
We only verify \eqref{E:Lemma4.0}, \eqref{E:Lemma4.0B} works in the same fashion.
Using that $(a+b+c)^2 \le 3(a^2+b^2+c^2)$, the left-hand side of \eqref{E:Lemma4.0} is at most
\begin{align}
	&6 \E\Big[ \Big(\sum_{y\in \cPc_n \cap Q_z  } \1{ \xi_r(y,\cP_n) \in I} \Big)^2 \Big] \label{E:Lemma4.1} \\
	&+ 3  \E\Big[ \Big(\sum_{y\in \cPc_n \setminus Q_z  } \1{ \xi_r(y,\cP_n) \in I} - \1{ \xi_r(y,\cP''_{n,z}) \in I}  \Big)^2 \Big] .\label{E:Lemma4.2}
\end{align}
First consider \eqref{E:Lemma4.1} and let $q\in\R_+$ such that $1/p + 1/q =1$. We have 
\begin{align*}
	&\E\Big[ \Big(\sum_{y\in \cPc_n \cap Q_z  } \1{ \xi_r(y,\cP_n) \in I} \Big)^2 \Big] \\
	&\le \E\Big[ \cPc_n (Q_z)^2 \1{ \xi_r(y,\cP_n) \in I \text{ for one } y \in \cPc_n \cap Q_z } \Big] \\
	&\le \E\Big[ \cPc_n (Q_z)^{2q} \Big]^{1/q}  \E\Big[ \sum_{ y \in \cPc_n \cap Q_z } \1{ \xi_r(y,\cP_n) \in I } \Big]^{1/p} \\
	&\lesssim |Q_z|^{2+1/p } \ \p(  \xi_r(0,\cP\cup \{0\}) \in I )^{1/p},
\end{align*}
where we use the H{\" o}lder inequality first and afterwards a simple upper bound on the maximum as well as the Slivnyak-Mecke formula. 

For the term in \eqref{E:Lemma4.2}, we proceed similarly. We use an upper bound as follows
\begin{align*}
	&\sum_{y\in \cPc_n \setminus Q_z  } \1{ \xi_r(y,\cP_n) \in I} - \1{ \xi_r(y,\cP''_{n,z}) \in I} \\
	&\le \sum_{y\in \cPc_n \setminus Q_z  } \1{ \xi_r(y,\cP_n) \in I} \1{ \xi_r(y,\cP''_{n,z}) \notin I} \\
	&\le \sum_{y\in \cPc_n \setminus Q_z  } \1{ \xi_r(y,\cP) \in I} \1{ R(y,\cP\cup\{y\}) > \|y-z\| - \frac{\sqrt{d}}{2} },
\end{align*}
where the last inequality holds because $B(y,r)\subseteq W_n$ for all $y\in \cPc_n$ and given that $\xi_r(y,\cP_n) \in I$ and $R(y,\cP\cup\{y\}) \le \|y-z\| - \sqrt{d}/2$, we have $\xi_r(y,\cP''_{n,z}) \in I$, too.

A similar lower bound is valid, namely, 
\begin{align*}
	&\sum_{y\in \cPc_n \setminus Q_z  } \1{ \xi_r(y,\cP_n) \in I} - \1{ \xi_r(y,\cP''_{n,z}) \in I} \\
	&\ge - \sum_{y\in \cPc_n \setminus Q_z  } \1{ \xi_r(y,\cP''_z) \in I} \1{ R(y,\cP\cup\{y\}) > \|y-z\| - \frac{\sqrt{d}}{2} }.
\end{align*}
Put $S_k(z) = B(z,k) \setminus B(z,k-1)$ as the difference between the Euclidean balls with radius $k$ resp. $k-1$ and center $z$.
Then up to a constant the term in \eqref{E:Lemma4.2} is at most
\begin{align}
	&\sum_{k_1, k_2= 1}^\infty \E\Bigg[ \sum_{y\in (\cPc_n \setminus Q_z )\cap S_{k_1}(z) } \sum_{y' \in (\cPc_n \setminus Q_z )\cap S_{k_2}(z)  } \1{ \xi_r(y,\cP) \in I} \1{ \xi_r(y',\cP) \in I} \nonumber \\
	&\qquad \1{ R(y,\cP\cup\{y\}) > \|y-z\| - \frac{\sqrt{d}}{2} }   \1{ R(y',\cP\cup\{y'\}) > \|y'-z\| - \frac{\sqrt{d}}{2} } \Bigg] \nonumber \\
	&\le 2 \sum_{ \substack{k_1, k_2= 1,\\ k_1 \le k_2} }^\infty \E\Bigg[ \cP( S_{k_1}(z)) \cP( S_{k_2}(z)) \ \1{ \xi_r(y,\cP) \in I \text{ for one } y\in (\cPc_n \setminus Q_z )\cap S_{k_1}(z) } \nonumber \\
	&\qquad\qquad\qquad \1{ R(y,\cP\cup\{y\}) > k_1 -  \frac{\sqrt{d}}{2} \text{ for one } y\in (\cPc_n \setminus Q_z )\cap S_{k_1}  } \nonumber \\
	&\qquad\qquad\qquad  \1{ R(y',\cP\cup\{y'\}) > k_2 -  \frac{\sqrt{d}}{2} \text{ for one } y'\in (\cPc_n \setminus Q_z )\cap S_{k_2}(z)  } \Bigg]\nonumber.
\end{align}
Let $w_1,w_2,q_1,q_2\in\R_+$ such that $1/w_1 + 1/w_2 + 1/q_1 + 1/ q_2 + 1/p =1$. Again we can apply the H{\"o}lder inequality first and afterwards crude upper bounds on the maximum as well as the Slyvniak-Mecke formula, to obtain an upper bound on the last double sum as follows
\begin{align*}
	& 2 \sum_{ \substack{k_1, k_2= 1,\\ k_1 \le k_2} }^\infty \E[ \cP(S_{k_1}(z))^{w_1}]^{1/w_1} \ \E[ \cP(S_{k_2}(z))^{w_2}]^{1/w_2}  \\
	&\qquad\qquad \E\Big[ \sum_{y \in (\cPc_n \setminus Q_z )\cap S_{k_1}(z) }  \1{ \xi_r(y,\cP) \in I } \Big]^{1/p}  \\
	&\qquad\qquad \E\Bigg[ \sum_{y\in (\cPc_n \setminus Q_z )\cap S_{k_1}(z)} \1{ R(y,\cP\cup\{y\}) > k_1 -  \frac{\sqrt{d}}{2}  } \Bigg]^{1/q_1}  \\
	&\qquad\qquad \E\Bigg[ \sum_{y'\in (\cPc_n \setminus Q_z )\cap S_{k_2}(z)} \1{ R(y',\cP\cup\{y'\}) > k_2 -  \frac{\sqrt{d}}{2}  } \Bigg]^{1/q_2}  \\
	&\lesssim \sum_{ \substack{k_1, k_2= 1,\\ k_1 \le k_2} }^\infty |S_{k_1}(z)|^{1+1/q_1+1/p} |S_{k_2}(z)|^{1+1/q_2} e^{-c_{stab}(k_1-\sqrt{d}/2)^{\as}/q_1 }  \\
	&\qquad\qquad \qquad\qquad e^{-c_{stab}(k_2-\sqrt{d}/2)^{\as}/q_2 } \ \p( \xi_r(0,\cP\cup\{0\})\in I )^{1/p}.
\end{align*}
Since the Lebesgue measure of the $S_{k}(z)$ grows polynomially in $k$, the last double sum is convergent. This completes the proof.
\end{proof}

\begin{lemma}\label{Lemma5}
Let $I\subseteq\R$ be an interval and let $Y$ be uniformly distributed on $B(0,\sqrt{d})$ and independent of $\cP$. There is a $C\in\R_+$ such that for all $z\in B_n$ with $Q_z\subseteq \Wc_n$ and $n\in\N$
\begin{align}
\begin{split}\label{E:Lemma5.0}
	&\E\Big[ \Big(\sum_{y\in \cPc_n \cap Q_z } \1{ \xi_r(y,\cP_n) \in I} \Big)^2 \Big] \\
	&\le C( \p(  \xi_r(0,\cP \cup \{0,Y\} ) \in I ) + \p(  \xi_r(0,\cP \cup \{0\} ) \in I ) ) .
\end{split}\end{align}
\end{lemma}
\begin{proof}
An application of the Slivnyak-Mecke formula yields on the one hand
\begin{align*}
		&\E\Big[ \sum_{y\in \cPc_n\cap Q_z} \1{ \xi_r(y,\cP_n) \in I}  \Big] =  \E[  \1{ \xi_r(0,\cP \cup \{0\} ) \in I} ] . 
\end{align*}
And on the other hand,		
\begin{align*}
		&\E\Big[ \sum_{y,y'\in \cPc_n\cap Q_z, y\neq y'} \1{ \xi_r(y,\cP_n) \in I} \1{ \xi_r(y',\cP_n) \in I} \Big] \\
		&= \int_{Q_z^2} \E[  \1{ \xi_r(y,\cP \cup \{y,y'\} ) \in I}  \1{ \xi_r(y,\cP \cup \{y,y'\} ) \in I}  ] \ \diff y \diff y' \\
		&\le \int_{Q_z^2} \E[  \1{ \xi_r(0,\cP \cup \{0,y'-y\} ) \in I} ] \ \diff y \diff y' \\
		&\lesssim  \E[  \1{ \xi_r(0,\cP \cup \{0,Y\} ) \in I} ].
\end{align*}
This completes the proof.
\end{proof}

\begin{lemma}\label{Lemma6}
Consider a bounded Borel set $\fA \subseteq \R^d$ which contains the origin 0. Let $u>0$, $k\in\N$ and let $Y_1,\ldots,Y_k$ be iid with a uniform distribution of $\fA$. Set $\lambda = |\fA|$. Then 
\[
	\p( R(0, \cP\cup\{0,Y_1,\ldots,Y_k\} ) \ge u) \le \lambda^{-k} \E[ \cP( \fA)^{2k}]^{1/2} \ \p( R(0, \cP\cup\{0\} ) \ge u)^{1/2}.
\]
\end{lemma} 
\begin{proof}
We partition $\cP$ into a local Poisson process $\wt \cP = \cP|_{\fA}$ and its complement $\cPd = \cP_{\R^d \setminus \fA}$. Let $X_1,X_2,\ldots$ be iid on $\fA$ with a uniform distribution. Set $p_j = e^{-\lambda} \lambda^j / j! $ and $\mX_j = \{X_1,\ldots,X_j\}$ for $j\in\N$. Then we obtain
\begin{align*}
	&\p( R(0, \cP\cup\{0,Y_1,\ldots,Y_k\} ) \ge u) \\
	 &= \sum_{j=k}^\infty p_{j-k} \ \p( R(0, \cPd \cup\{0,\mX_j\} ) \ge u ) \\
	&=\sum_{j=k}^\infty p_j \ \frac{p_{j-k}}{p_j} \p( R(0, \cPd \cup\{0,\mX_j \} ) \ge u )  \\
	&\le \left(\sum_{j=k}^\infty p_j \ \Big(\frac{p_{j-k}}{p_j}\Big)^2 \right)^{1/2} \left( \sum_{j=k}^\infty p_j \ \p( R(0, \cPd \cup\{0,\mX_j\} ) \ge u )^2 \right)^{1/2} \\
	&\le \lambda^{-k} \ \left(\sum_{j=k}^\infty p_j \ j^{2k} \right)^{1/2} \ \left( \sum_{j=0}^\infty p_j \ \p( R(0, \cPd \cup\{0,\mX_j\} ) \ge u ) \right)^{1/2} \\
	&\le \lambda^{-k} \E[ \cP( \fA)^{2k}]^{1/2} \ \p( R(0, \cP\cup\{0\} ) \ge u)^{1/2}.
\end{align*}
This shows the statement.
\end{proof}

\begin{lemma}\label{L:LILTool}
Let $(X_n)_n$ be a sequence of random variables satisfying 
$$
(\log n )^{-\alpha} n^{-\beta} \| \max_{k\le n} |X_k| \|_q \le C (\log n)^{-(1+\delta)}
$$
for constants $q\ge 1$ and $C,\alpha,\beta, \delta \in\R_+$ (which are independent of $n$).

Then $\limsup_{n\to\infty} (\log n )^{-\alpha} n^{-\beta} |X_n| \le \wt C$ $a.s.$ for some $\wt C\in\R_+$.
\end{lemma}
\begin{proof}
Clearly, $\sum_{j\in\N} (\log (2^j) )^{-\alpha} 2^{- j \beta} \| \max_{k\le 2^j} |X_k| \|_q < \infty$. 

Hence, $ \max_{k\le 2^j} |X_k|  = O_{a.c.}[ (\log (2^j) )^{\alpha}  (2^j)^{\beta} ]$. In particular, there is a constant $C^*\in\R_+$ which satisfies
$$
		\limsup_{j\to\infty} \	(\log (2^j) )^{-\alpha}  (2^j)^{-\beta} \max_{k\le 2^j} |X_k| \le C^* \quad a.s.
$$
Let $n\ge 4$ and write $\ol j$ for the largest integer such that $2^{\ol j -1 } < n \le 2^{\ol j}$. Then $\ol j \ge 2$ and
\begin{align*}
	(\log n )^{-\alpha} n^{-\beta} \max_{k\le n} |X_k| &\le \frac{ (\log (2^{\ol j}) )^{\alpha}}{ (\log n )^{\alpha}} \ \frac{ 2^{ \beta \ol j}}{n^\beta} \ \frac{ \max_{k\le 2^{\ol j}} |X_k| }{ (\log (2^{\ol j} ) )^{\alpha} 2^{\beta \ol j} }   \\
	&\le \frac{\ol j ^\alpha \ 2^\beta}{(\ol j -1 )^{\alpha} }  \frac{ \max_{k\le 2^{\ol j}} |X_k| }{ (\log (2^{\ol j} ) )^{\alpha} 2^{\beta \ol j} }   \le 2^{\alpha+\beta} \frac{ \max_{k\le 2^{\ol j}} |X_k| }{ (\log (2^{\ol j} ) )^{\alpha} 2^{\beta \ol j} }  . 
\end{align*}
Consequently,
$$
	\limsup_{n\to\infty} \ (\log n )^{- \alpha} n^{-\beta} |X_n| \le  \limsup_{n\to\infty} \ (\log n )^{-\alpha} n^{-\beta} \max_{k\le n } |X_k| \le 2^{\alpha+\beta} C^* < \infty
$$
with probability 1.
\end{proof}

\end{document}